\documentclass{article}
\usepackage[margin=1in]{geometry}

\usepackage{amsmath}
\usepackage{amssymb}
\usepackage{mathtools}
\usepackage{amsfonts}
\usepackage{latexsym}
\usepackage{framed}
\usepackage{amsthm}
\usepackage{color}

\newtheorem{thm}{Theorem}
\numberwithin{thm}{section}
\newtheorem{prop}[thm]{Proposition}
\newtheorem{cor}[thm]{Corollary}
\newtheorem{asu}[thm]{Assumption}

\theoremstyle{definition}
\newtheorem{example}[thm]{Example}

\begin{document}
\date{}
\title{Forward sensitivity analysis for contracting\\stochastic systems}
\author{Thomas Flynn\footnote{Work performed at Department of Computer Science, Graduate Center of CUNY, New York, NY, USA. Current address: Brookhaven National Laboratory, Upton, NY, USA. }}
{\let\newpage\relax\maketitle}
\begin{abstract}
In this work we investigate gradient estimation for a class of contracting stochastic systems on a continuous state space.
\empty
We find conditions on the one-step transitions, namely differentiability and contraction in a Wasserstein distance, that guarantee differentiability of stationary costs. 
\empty
Then we show how to estimate the derivatives, deriving an estimator that can be seen as a generalization of the forward sensitivity analysis method used in deterministic systems.
\empty
We apply the results to examples, including a neural network model.
\end{abstract}
\section{Introduction}
Stationary gradient estimation starts with a Markov kernel 
$P$
that depends on a parameter 
$\theta$.
Given a cost function 
$e$
defined on the states of the Markov chain, and assuming ergodicity of the process, the problem is to estimate the derivative of the average cost, at stationarity, with respect to the parameter
$\theta$.
That is, setting
$\pi_{\theta}$
to the stationary measure of 
$P_{\theta}$,
the problem is to estimate
$$
\tfrac{\partial }{\partial \theta}
\int_{X}e(x)\,\mathrm{d}\pi_{\theta}(x)
$$
In this work we investigate an approach to this problem based on forward sensitivity analysis, an algorithm used for estimating sensitivities in deterministic systems. We review this now to show the main idea. 

Consider a continuous state space $X\subseteq \mathbb{R}^{n_{X}}$ and a parameter space $\Theta \subseteq \mathbb{R}^{n_{\Theta}}$. Let
$f : X \times \Theta \to X$ 
be such that 
$f(\cdot,\theta)$
is a contraction mapping on $X$
for all values of $\theta$.
Then $f$ has a unique fixed-point
$x^{*}(\theta)$ 
for each $\theta \in \Theta$.
With further conditions on the differentiability of $f$,
it holds that $x^{*}$ is differentiable in $\Theta$. 
The problem is to estimate 
\begin{equation}\label{fp-der-det}
\tfrac{\partial }
     {\partial \theta}
(e \circ x^{*})(\theta)
\end{equation}
Let 
$M = L(\mathbb{R}^{n_{\Theta}},\mathbb{R}^{n_{X}})$, the space of linear maps from $\mathbb{R}^{n_{\Theta}}$ to $\mathbb{R}^{n_X}$. Define the map 
$T : X\times M \times \Theta \to X \times M $
by
$$
T( (x,m),\theta) = 
\left( 
  f(x,\theta), 
  \tfrac{\partial f}{\partial x}(x,\theta)m 
  +
  \tfrac{\partial f}{\partial \theta}(x,\theta)
\right)
$$
Using assumptions on the derivatives and contraction properties of $f$,
one can show that 
$T(\cdot,\theta)$
is also a contraction,
for a suitable metric on $X\times M$. 
Denoting by $(x^*,m^*)$ the fixed-point of $T$ at $\theta$, 
it can be proven that the derivative of the fixed-point cost is
$$
\tfrac{\partial }{\partial \theta}(e \circ x^{*})(\theta) = 
\tfrac{\partial e}{\partial x}(x^{*})m^{*}
$$
Based on this, to approximately compute (\ref{fp-der-det}) one can iterate $T$ to obtain a pair 
$(x,m)$
near
$(x^{*},m^{*})$,
and then prepare the gradient estimate by computing 
$\frac{\partial e}{\partial x}(x)m$.
For more background on forward sensitivity analysis we refer the reader to \cite{griewank-book}, Chapter 15.

This work considers the method in the probabilistic setting.
Let $P_{\theta}$ take the form 
$$
(P_{\theta}e)(x) = \int_{\Xi}e(f(x,\xi,\theta))\,\mathrm{d}\nu(\xi)
$$
for a probability space $(\Xi,\Sigma,\nu)$
and a function 
$f:X\times \Xi\times \Theta \to X$.
We find that if certain contraction and
differentiability conditions are satisfied, then
\begin{equation}\label{der-comp-id}
  \tfrac{\partial}{\partial \theta}
  \int_{X}e(x)\,\mathrm{d}\pi_{\theta}(x) 
  = 
  \int_{X\times M}
  \tfrac{\partial e}{\partial x}(x)m 
  \,\mathrm{d}\gamma_{\theta}(x,m)
\end{equation}
where 
$\gamma_{\theta}$
is the stationary measure on 
$X\times M$
of the recursion
\begin{subequations}
\begin{flalign}
\label{optproc1}
x_{n+1} &= f(x_n,\xi_{n+1},\theta) \\
\label{optproc2} 
m_{n+1} &= \tfrac{\partial f}{\partial x}(x_n,\xi_{n+1},\theta)m_n + \tfrac{\partial f}{\partial \theta}(x_n,\xi_{n+1},\theta)
\end{flalign}
\end{subequations}
where the $\xi_{n}$ form an i.i.d. sequence of $\nu$-distributed random variables.
There are several challenges associated with this.
The first is to extend the contraction framework to include probabilistically interesting systems.
The contraction framework should enable us to show convergence of the forward sensitivity process (\ref{optproc1}, \ref{optproc2}) as well as the underlying process. The second challenge is to show correctness of the procedure. 

A simple case of our main result can be stated as follows. In the statement of this theorem and throughout the article, a function is said to be $C^{1}$ if it is continuously differentiable and the function is $C^{2}$ if it is twice continuously differentiable. For a function $h$ defined on a set $X$ and taking values in a normed space, $\|h\|_{\infty} = \sup_{x\in X}\|h(x)\|$.

\begin{thm}
Let the function 
$f$
and the probability space $(\Xi,\Sigma,\nu)$ be such that
\begin{enumerate}
   \renewcommand{\theenumi}{\roman{enumi}}
\item 
   $\int_{\Xi}\|f(x,\xi,\theta)\|^{2}\,\mathrm{d}\nu(\xi) <\infty$ for all $(x,\theta) \in X\times\Theta$,
\item 
  $(x,\theta) \mapsto f(x,\xi,\theta)$ is a $C^{2}$ function for each $\xi\in\Xi$,
\item For $0< i+j\leq 2$, the functions 
  $L_{X^{i},\Theta^{j}}(x,\theta) 
  = 
  \int_{\Xi}
  \|\tfrac{\partial^{i+j}f}{\partial x^{i}\partial \theta^{j}}(x,\xi,\theta)\|^{2}
  \,\mathrm{d}\nu(\xi)$
are continuous and bounded on $X\times \Theta$, and in particular, $\sup_{(x,\theta)}L_{X}(x,\theta)  <1$.
\end{enumerate}
Then the forward sensitivity process (\ref{optproc1}, \ref{optproc2}) converges weakly to a stationary measure $\gamma_{\theta}$, and equation (\ref{der-comp-id}) holds for those $e:X\to\mathbb{R}$ that are $C^{2}$ with $\|\frac{\partial e}{\partial x}\|_{\infty} + \|\frac{\partial^{2} e}{\partial x^{2}}\|_{\infty} < \infty$.
\end{thm}
The full version, stated below in Theorem \ref{mainthm} relaxes the assumptions. In the general version the various bounds are assumed to hold with respect to a Finsler structure.
\subsection{Overview of main results}
First the contraction framework is introduced.
Second, criteria for differentiability of the stationary costs are presented. The third component is a set of conditions on the function 
$f$ 
that let us apply the abstract result on stationary differentiability, establish convergence of the sensitivity process 
$(x_n,m_n)$,
 and allow us to show that 
equation (\ref{der-comp-id})
 holds. Finally, we consider an application to neural networks.
\subsubsection{Contraction framework.}
Given a matrix valued function $A(x)$ and a norm $\|\cdot\|$ on $\mathbb{R}^{n_{X}}$, we consider the following ergodicity condition
\begin{equation}\label{avg-ctr-lyap}
 \sup_{x\in X}
 \left(
   \int_{\Xi}
   \|
   A(f(x,\xi))
   \tfrac{\partial f}{\partial x}(x,\xi)A(x)^{-1}
   \|^{p}
   \,\mathrm{d}\nu(\xi)
   \right)^{1/p}
   < 1
\end{equation}
The object inside the norm is the composition of the three linear maps $A(f(x,\xi))$, $\frac{\partial f}{\partial x}(x,\xi)$ and $A(x)^{-1}$, and the norm in this inequality is that induced by $\|\cdot\|$ on the space of linear maps $L(\mathbb{R}^{n_{X}},\mathbb{R}^{n_{X}})$.
 Formally, the map 
 $(x,u) \mapsto \|A(x)u\|$
 defines a Finsler structure on the space $X$, 
 which induces a metric $d_A$ on $X$. 
 This is extended to a metric on probability measures using the Wasserstein distance $d_{p,A}$.
 The condition (\ref{avg-ctr-lyap}) implies the Markov kernel $P$ is a contraction mapping for this distance.
This is developed in Section \ref{sect:contraction-framework}. In Section \ref{sect:interconnection-of-contractions} we consider interconnections of contracting systems,
obtaining sufficient conditions for both feedback and hierarchical combinations of contracting systems to again be contracting. This is useful to analyze the forward sensitivity process,
as it exhibits a hierarchical structure.

\subsubsection{Stationary differentiability.}
In Section 
\ref{sect:stationary-differentiability}
we give abstract conditions for stationary differentiability, using a variant of the proof technique in \cite{heidergott2003taylor}. 
The equation
\begin{equation}\label{grad-eqn}
  l = l P_{\theta} + \pi_{\theta}\tfrac{\partial}{\partial \theta}P_{\theta}
\end{equation}
is shown to have a unique solution in the variable $l$, and this $l$ is shown to evaluate the stationary derivatives, meaning
$
l(e) = \tfrac{\partial}{\partial \theta}\int_{X}e(x)\,\mathrm{d}\pi_{\theta}(x).
$
While similar formulas have been recovered by other authors
(see \cite{vazquez1992estimation, heidergott2003taylor, pflug-mvd}) we rederive this using assumptions that are relevant for the smooth systems we are interested in. 

\subsubsection{Gradient estimation.}
To study  the forward sensitivity process we define an appropriate metric on the space $X\times M$ and prove a pointwise contraction inequality for the joint system (\ref{optproc1}, \ref{optproc2}) in this distance. This is used together with a Lyapunov function for the joint system to establish ergodicity of the sensitivity process.
This is done in Section \ref{sect:gradient-estimation}. It is then established that the functional
$
e
\mapsto
\int_{X\times M}
\tfrac{\partial e}{\partial x}(x)m 
\,\mathrm{d}\gamma_{\theta}(x,m)$
verifies equation (\ref{grad-eqn}).
We conclude that equation (\ref{der-comp-id}) holds for the class of cost functions. 
\\\\
Before formally stating the assumptions and  main results, we introduce some notation and conventions. For a function $f:X\to \mathbb{R}^{n}$  where $X \subseteq \mathbb{R}^{m}$, we denote by $\frac{\partial f}{\partial x}(x_0)$ the derivative of $f$ with respect to $x$ at the point $x_0$, and for a vector $u \in \mathbb{R}^{m}$, we denote by $\frac{\partial f}{\partial x}(x_0)u$  the $\mathbb{R}^{n}$-valued result of applying this linear map to the vector $u$. The second derivative of $f$ with respect to $x$ is $\frac{\partial^{2} f}{\partial x^{2}}$, and $\frac{\partial^{2}f}{\partial x^{2}}(x_0)[u,v]$ refers to the $\mathbb{R}^{m}$-valued result of applying this bilinear map to the arguments $u,v$. Given norms $\|\cdot\|_{X}$ and $\|\cdot\|_{Y}$ on the space $\mathbb{R}^{m}$ and $\mathbb{R}^{n}$, recall that the norm of a linear map $E:\mathbb{R}^{n} \to \mathbb{R}^{m}$ is $\|E\| = \sup_{\|u\|_{X} =1}\|Eu\|_{Y}$. For a bilinear map $F$ defined on $\mathbb{R}^{n} \times \mathbb{R}^{m}$ and taking values in a third space with norm $\|\cdot\|_{Z}$, the norm is $\|F\| = \sup_{\|u\|_{X}=\|v\|_{Y} = 1}\|F[u,v]\|_{Z}$. Given two linear maps $E$ and $F$, their direct sum is the linear map
$(E \oplus F)(u,v) = (Eu,Fv)$. For reference the appendix contains a summary of notations and definitions of spaces used throughout the paper.

\begin{asu}\label{asu:lipVA}
The set $X$ is a closed, convex subset of $\mathbb{R}^{n_X}$, and $\mathbb{R}^{n_{X}}$ carries a norm $\|\cdot\|_X$.
The function $A:X\to L(\mathbb{R}^{n_{X}},\mathbb{R}^{n_{X}})$ is continuous, such that each $A(x)$ is invertible, and
   $\sup_{x\in X}\|A(x)^{-1}\|_{X} < \infty$.
\end{asu}
We will require differentiability and integrability of $f$:
\begin{asu}\label{asu:differentiable}
For an open set $\Theta \subseteq \mathbb{R}^{n_{\Theta}}$, the function
$f: X \times \Xi \times \Theta \to X$
satisfies
\begin{enumerate}
   \renewcommand{\theenumi}{\roman{enumi}}
\item  
  $\xi \mapsto d_A(x,f(x,\xi,\theta))^{2}$
  is $\nu$-integrable for all 
  $(x,\theta) \in X\times\Theta$,
\item 
  $(x,\theta) \mapsto f(x,\xi,\theta)$ is twice continuously differentiable ($C^{2}$) for each $\xi\in\Xi$.
\end{enumerate}
\end{asu}
We also require some bounds on $P$
as a function of $\theta$, formulated with the help of a function $B(x)$ 
taking values in the invertible 
$n_{\Theta} \times n_{\Theta}$
matrices. 
\begin{asu}\label{asu:lipVB}
$\mathbb{R}^{n_{\Theta}}$ has a norm $\|\cdot\|_{\Theta}$. The function $B:X\to L(\mathbb{R}^{n_{\Theta}},\mathbb{R}^{n_{\Theta}})$ takes values in the invertible linear maps, and $x \mapsto \|B(x)\|_{\Theta}$ is a $d_A$-Lipschitz function.
\end{asu}
For an example when 
Assumption \ref{asu:lipVB}
is satisfied, consider the following. 
Let 
$g:X \to \mathbb{R}_{\geq 0}$
be a function  that is Lipschitz continuous 
with respect to the underlying norm $\|\cdot\|_{X}$ on $X$. 
Then use
$A(x) = \exp(g(x))I_{n_{X}}$ and $B(x) = \exp(g(x))I_{n_{\Theta}}$, where $I_n$ is the $n\times n$ identity matrix. 
Of course, the assumption always holds when
$B(x) = I_{n_{\Theta}}$.

The next assumptions relate to the contraction property of $P$ and the differentiability properties of $P_{\theta}e$.
Before continuing we define several norms derived from $A$ and $B$.
At each $x \in X$ the matrix $A(x)$ defines a norm 
$\|\cdot\|_{A(x)}$ 
on
$\mathbb{R}^{n_{X}}$
by
$\|u\|_{A(x)} = \|A(x)u\|$.
and $B(x)$
defines a norm on
$\mathbb{R}^{n_{\Theta}}$
by
$\|v\|_{B(x)} = \|B(x)v\|$.
These extend to norms on the various linear spaces.
For example, if 
$l \in L(\mathbb{R}^{n_{X}},\mathbb{R})$
then
$
\|l\|_{A(x)}
=
\|l A(x)^{-1}\|
$. 
For a bilinear map 
$
Q
\in
L(\mathbb{R}^{n_{X}},\mathbb{R}^{n_{X}};\mathbb{R})
$
we can write
$
\|Q\|_{A(x),A(x)} = 
\|Q(A(x)^{-1}\oplus A(x)^{-1})\| 
$.
Further extend this to functions from 
$X$
into the linear spaces by taking supremums, e.g. if
$h:X\to L(\mathbb{R}^{n_{\Theta}},\mathbb{R})$
then
$\|h\|_{B} = \sup_{x}\|h(x)\|_{B(x)}$.
For the case of a real-valued $h:X\to\mathbb{R}$, let $\|h\|_{A} = \sup_{x}\frac{|h(x)|}{1+d_{A}(x,x_0)}$, where $x_0$ is an arbitrary basepoint in $X$.

We introduce the space of cost functions $\mathcal{E}^{2}$:
$$
\mathcal{E}^{2} =
\{
h: X\to \mathbb{R} \mid h \text{ is } C^{2}
\text{ and }
\|h\|_{A} + \|\tfrac{\partial h}{\partial x}\|_{A} 
+ 
\|\tfrac{\partial^{2}h}{\partial x^{2}}\|_{A,A} < \infty
\}$$
On
$\mathcal{E}^{2}$
we put the norm 
\begin{equation}\label{e2-norm}
\|h\|_{\mathcal{E}^{2}} 
= 
\|h\|_{A} + 
\|\tfrac{\partial h}{\partial x}\|_{A} 
+ 
\|\tfrac{\partial^{2} h}{\partial x^{2}}\|_{A,A}
\end{equation}
We consider bounds on the derivatives of $f$
formulated using the following functions:
    $$ L_X(x,\theta) = \left(\int_{\Xi}
      \|A(f(x,\xi,\theta))
      \tfrac{\partial f}{\partial x}(x,\xi,\theta)
      A(x)^{-1}
      \|^{2}
      \, \mathrm{d}\nu(\xi)\right)^{1/2}$$
    $$L_{\Theta}(x,\theta)= \left(\int_{\Xi}
      \|A(f(x,\xi,\theta))
      \tfrac{\partial f}{\partial \theta}(x,\xi,\theta)
      B(x)^{-1}
      \|^{2}
      \, \mathrm{d}\nu(\xi)\right)^{1/2}$$
    $$ L_{X^{2}}(x,\theta) = \int_{\Xi}
    \|A(f(x,\xi,\theta))
    \tfrac{\partial^{2} f}{\partial x^{2}}(x,\xi,\theta)
    \left(A(x)^{-1} \oplus A(x)^{-1}\right)
    \|
    \, \mathrm{d}\nu(\xi)$$
    $$L_{\Theta^{2}}(x,\theta)
    =
    \int_{\Xi}
    \|A(f(x,\xi,\theta))
    \tfrac{\partial^{2} f}{\partial \theta^{2}}(x,\xi,\theta)
    \left(B(x)^{-1} \oplus B(x)^{-1}\right)
    \|
    \, \mathrm{d}\nu(\xi)$$
    $$
    L_{X,\Theta}(x,\theta)
    =
    \int_{\Xi}
    \|A(f(x,\xi,\theta))
    \tfrac{\partial^{2} f}{\partial x\partial \theta}
    (x,\xi,\theta)
    \left(A(x)^{-1} \oplus B(x)^{-1}\right)
    \|
    \, \mathrm{d}\nu(\xi)
    $$
\begin{asu}\label{asu:f}The functions $L_{X^{i},\Theta^{j}}$ satisfy
\begin{enumerate}
   \renewcommand{\theenumi}{\roman{enumi}}
 \item  The various functions 
    $L_{X^{i},\Theta^{j}}$
    are continuous on $X\times \Theta$,
\item There is a $K_X \in [0,1)$ such that
  $\sup\limits_{(x,\theta)\in X\times\Theta}L_{X}(x,\theta) \leq K_X$, \label{part-ctr-of-l-asu}
\item For $0 < i+j\leq 2$, there are
    $K_{X^i,\Theta^j}$ such that
    $\sup\limits_{(x,\theta)\in X\times\Theta}L_{X^{i},\Theta^{j}}(x,\theta) \leq K_{X^{i},\Theta^{j}}$. \label{part-other-asu}
\end{enumerate}
\end{asu}
Using these assumptions and definitions, we can now state the main result.
\begin{thm}\label{mainthm}
  Let Assumptions \ref{asu:lipVA}, \ref{asu:differentiable},  \ref{asu:lipVB},
  and \ref{asu:f} be satisfied.
  Let $\theta$ be an arbitrary point of $\Theta$. Then the forward sensitivity process (\ref{optproc1}, \ref{optproc2}) possesses a unique stationary measure $\gamma_{\theta}$
  and
  for any
  $e\in \mathcal{E}^{2}$
  equation (\ref{der-comp-id}) is valid.
  Furthermore, if the variables $(x_1,m_1)$ satisfy the integrability condition
  $\mathbb{E}[d_A(x_0,x_1) + \|A(x_1)m_1\|] < \infty$
  for an arbitrary basepoint $x_0$,  then
  $\mathbb{E}[\tfrac{\partial e}{\partial x}(x_n)m_{n}]
  \rightarrow
  \tfrac{\partial}{\partial \theta}\int_{X}e(x)\, \mathrm{d}\pi_{\theta}(x)$
  as $n\rightarrow \infty$.
\end{thm}
\subsubsection{Neural network application.}
In Section \ref{sect:example} 
two examples are considered.
The first involves neural networks.
In neural networks, a central problem is to compute derivatives
of cost functionals with respect to network parameters (weights on the connections between nodes). 
We are concerned with long-term average cost problems, a type of problem that is relevant when a network has cycles.
The back-propagation algorithm for calculating derivatives \cite{rumelhart1986learning}, originally formulated for a continuous state-space model with a finite horizon objective, is also valid for calculating gradients in long-term average cost problems under contraction assumptions \cite{pineda88}.  Our contribution addresses the long-term average cost problem for continuous stochastic networks.

The example system consists of a network with weights on connections between units. 
At each step every node updates its value based on the values of its neighbors, but only a random subset of possible connections are activated, leading to a stochastic process. 
We find contraction conditions based on a sparsity coefficient,  and verify that stochastic forward sensitivity analysis can be used to calculate the derivative of stationary costs.  We present a second example to illustrate using a non-trivial metric on the underlying system. We finish with a discussion in Section \ref{sect:discussion}.

\section{Contraction framework}
\label{sect:contraction-framework}
We describe a class of metrics on Euclidean space that form the basis for the subsequent discussion of contraction. 
These metrics are defined by minimizing a length functional, and 
form a subclass of the Finsler metrics. Then we present ergodicity conditions which rely on pointwise contraction estimates involving such metrics.

Let $X$ be a closed convex subset of the Euclidean space $\mathbb{R}^{n}$ and let $[x\leadsto y]$ be the set of piecewise $C^{1}$ curves from $x$ to $y$.
Given a norm $\|\cdot\|$ on $\mathbb{R}^{n}$ and 
 a function $x \mapsto A(x)$ taking values in the invertible $n\times n$ matrices, one can define a metric on $X$ as follows. 
\begin{prop}\label{conditions-for-metric}
  Let $\|\cdot\|$ be a norm on $\mathbb{R}^{n}$ and let
  $x \mapsto A(x)$ be a continuous function that assigns to each 
  $x\in X$ an invertible linear map $A(x)$ on 
  $\mathbb{R}^{n}$, in such a way that
  $\sup_{x\in X} \|A(x)^{-1}\| < \infty$. 
  For a piecewise $C^{1}$ curve 
  $\gamma:[\gamma_{s},\gamma_{e}] \to X$,
  define
  $L(\gamma) = \int_{\gamma_{s}}^{\gamma_{e}}\|A(\gamma(t))\gamma'(t)\|\,\mathrm{d}t$.
  Then the function
  $d_{A}(x,y) = \inf_{\gamma \in [x\leadsto y]}L(\gamma)$
  defines a metric on $X$ compatible with the Euclidean topology, 
  and $(X,d_A)$ is complete.
\end{prop}
\begin{proof}
See the appendix.
\end{proof}
For instance taking 
$A=I_{n}$
one recovers the norm 
$d_A(x,y) = \| x-y\|$. 
Using 
$A(x) = V(x)I_{n}$ for real-valued function $V$
means a cost $V(x)$ is assigned for going through each point $x$.
Using a general matrix allows the cost for traveling through each point $x$ to also depend on the direction of the path at the point.
For a function $e:X\to\mathbb{R}$ we let $\|e\|_{Lip(A)}$ be the Lipschitz constant of a function $e:X\to\mathbb{R}$ with respect to the metric $d_A$. When the metric $d_A$ is clear we will just write $\|e\|_{Lip}$.

The collection of Borel probability measures on $X$ is denoted 
$\mathcal{P}(X)$. We denote by $\mu(e)$ the expectation of $e$ under $\mu$. That is, $\mu(e) = \int_{X}e(x)\,\mathrm{d}\mu(x)$. 
For a number $k$ we let $\mathbb{R}_{\geq k}$ be the set $\left\{ x \in \mathbb{R} \mid x \geq k \right\}$. For a probability measure $\mu$ and $p\geq 1$ we write 
$\|V\|_{L^{p}(\mu)} 
= 
\left(\int_{X}\|V(x)\|^{p}\, \mathrm{d}\mu(x)\right)^{1/p}$.
Given a function $V:X\to\mathbb{R}_{\geq 0}$
 the space
$\mathcal{P}_{p,V}(X)$ is defined to be
 all Borel measures $\mu$ on $X$ which can integrate $V^{p}$:
$$\mathcal{P}_{p,V}(X) = 
\left\{ 
  \mu \in \mathcal{P}(X) 
  \, \middle| \, 
  \int_{X}V(x)^{p}\,\mathrm{d}\mu(x) 
  < \infty \right\}$$
Given a Markov kernel $P$, we denote the image of measure $\mu$ under $P$ by $\mu P$. That is, $(\mu P)(A) = \int_{X}P(x,A)\,\mathrm{d}\mu(x)$. For $V : X \to\mathbb{R}_{\geq 1}$, let $\|e\|_{V}=\sup\limits_{x\in X}\tfrac{|e(x)|}{V(x)}$. We say that $V:X\to\mathbb{R}_{\geq 1}$
is a $p$-Lyapunov function for $P$ if $V$ has compact sublevel sets and there exists numbers $\beta\in[0,1),\, K\geq 0$ so that $\left(PV^{p}(x)\right)^{1/p} \leq \beta V(x) + K$ for all $x$. 
A measure $\mu\in\mathcal{P}(X\times X)$ is a coupling of $\mu_1$ and $\mu_2$ if
$\mu(A\times X) =\mu_1(A)$ and $\mu(X\times A) = \mu_2(A)$ for each measurable set $A$.
We define
$\Gamma(\mu_1,\mu_2)$
to be the set of all couplings of $\mu_1$ and $\mu_2$.

Let the Markov kernel $P$ have an explicit representation as 
\begin{equation}\label{explicit-rep}
(Pe)(x)= \int_{\Xi}e(f(x,\xi))\,\mathrm{d}\nu(\xi)
\end{equation}
for a measurable function 
$f:X\times \Xi \to X$
and a probability space
$(\Xi,\Sigma,\nu)$.
In this section we present two separate conditions for the ergodicity of a Markov kernel given in the form (\ref{explicit-rep}). The first, Proposition \ref{contraction-weak}, is weaker and is  used to show convergence of the forward sensitivity system (consisting of the variables $x_n,m_n$). Proposition \ref{contraction-strong} relies on a stronger set of assumptions and is used to establish differentiability of the stationary costs. Both results utilize the following pointwise estimate of Proposition \ref{contraction-ptwise}. 

In this proposition, and throughout the paper, we consider a differentiable function defined on a closed subset $X$ of Euclidean space. In case $X$ is a strict subset of the space, we assume $f$ is the restriction of a function $\overline{f}$ that is defined and differentiable on an open set $U$ containing $X$. In this way there is no ambiguity in defining the derivative of $f$ at each point of $X$.
\begin{prop}\label{contraction-ptwise}
  Let $P$ be of the form (\ref{explicit-rep})
  where 
  \begin{enumerate}
    \renewcommand{\theenumi}{\roman{enumi}}
  \item $x \mapsto f(x,\xi)$ is $C^{1}$ for each $\xi\in \Xi$, \label{differentiability-assumption}
\item \label{contract-sufficient}
$\sup\limits_{x\in X}
\sup\limits_{u \in \mathbb{R}^{n}:\|u\|=1}
\left(\displaystyle\int_{\Xi}
  \|A(f(x,\xi))\tfrac{\partial f}{\partial x}(x,\xi)A^{-1}(x)u\|^{p}
  \,\mathrm{d}\nu(\xi)\right)^{1/p}
\leq
\alpha,
$
\end{enumerate}
for some $\alpha \geq 0$.
Then for any $x_1,x_2 \in X$ we have
  \begin{equation}\label{common-random-numbers-coupling-contraction}
  \left(
    \int_{\Xi}
    d_A\Big( f(x_1,\xi),f(x_2,\xi)\Big)^{p}
    \,\mathrm{d}\nu(\xi)
  \right)^{1/p} \leq \alpha d_A(x_1,x_2).
  \end{equation}
 \end{prop}
\begin{proof}
  Let $x_1\neq x_2$ be points of $X$, let 
  $\epsilon>0$
  and let 
  $\gamma : [0,T] \to X$
  be a piecewise $C^{1}$ path 
  from $x_1$ to $x_2$ such that 
  $L(\gamma) \leq d_{A}(x_1,x_2) + \epsilon$. 
  We further assume that $\gamma$ is parameterized by arc length.
  For our definition of length this means
  $\|A(\gamma(t))\gamma'(t)\| = 1$
  for all $t$ and that $T = L(\gamma)$.
  Since 
  $t \mapsto f(\gamma(t),\xi)$
  defines a curve from 
  $f(x_1,\xi)$ to $f(x_2,\xi)$ we have
  \begin{align*}
    &\left(\int_{\Xi}
      d_A(f(x_1,\xi),f(x_2,\xi))^{p} 
      \,\mathrm{d}\nu(\xi)
    \right)^{1/p}\\ 
    &\quad\quad\leq 
    \left(\int_{\Xi}
      \left(\int_{0}^{T}
        \|
        A(f(\gamma(t),\xi))
        \tfrac{\partial f}{\partial x}(x,\xi)
        \gamma'(t)\| 
        \,\mathrm{d}t\right)^{p}
      \,\mathrm{d}\nu(\xi)\right)^{1/p}  \\
    &\quad\quad\leq
    L(\gamma)^{(p-1)/p}
    \left(\int_{\Xi}
      \int_{0}^{T}
      \|
      A(f(\gamma(t),\xi))\tfrac{\partial f}{\partial x}(x,\xi)\gamma'(t)
      \|^{p} 
      \,\mathrm{d}t \,\mathrm{d}\nu(\xi)
    \right)^{1/p} 
\end{align*}
In the first step the definition of length was applied.
Then Jensen's inequality was used together with the fact that $L(\gamma) = T$.
Next, note the integrand in the final expectation is of the form 
$(t,\xi) \mapsto g(t,\xi)$
where $g$ is non-negative, continuous in $t$ for each $\xi$, and measurable in $\xi$ for each $t$. Then we may interchange the integrals, yielding
\begin{align*}
    &\quad\quad=
      L(\gamma)^{(p-1)/p}
     \left(
      \int_{0}^{T}
      \int_{\Xi}
      \|A(f(\gamma(t),\xi))
      \tfrac{\partial f}{\partial x}(x,\xi)
      \gamma'(t)\|^{p} 
      \,\mathrm{d}\nu(\xi) \,\mathrm{d}t
    \right)^{1/p}
\end{align*}
Using the identity 
$A(\gamma(t))^{-1}A(\gamma(t))\gamma'(t) = \gamma'(t)$ and the assumption on 
$\tfrac{\partial f}{\partial x}$ we get
\begin{align*}
  &\leq
    L(\gamma)^{(p-1)/p}
     \left(
      \int_{0}^{T}
    \alpha^{p}\|A(\gamma(t))\gamma'(t)\|^{p} 
     \,\mathrm{d}t
    \right)^{1/p}
\end{align*}
Then since $\gamma$ is parameterized by arc length,
\begin{align*}
    &=
    L(\gamma)^{(p-1)/p}\alpha L(\gamma)^{1/p}
    \leq
    \alpha d_A(x_1,x_2) + \alpha\epsilon
  \end{align*}
  As $\epsilon>0$ was arbitrary, the result follows.
\end{proof}
If a tuple 
$\{(\Xi,\Sigma,\nu),f,(\|\cdot\|,A)\}$
satisfies the conditions of Proposition \ref{contraction-ptwise} for some $\alpha < 1$,
we say that a \textit{pointwise $p$-contraction inequality} holds for the process.

Combining this with the assumption that the system carries a Lyapunov function yields the following ergodicity result. 
\begin{prop}\label{contraction-weak}
  Let the assumptions of Proposition \ref{contraction-ptwise} hold for $p\geq 1$ and $\alpha < 1$, and assume  there is a $p$-Lyapunov function $V$ for $P$.
Then $P$ has a unique invariant measure $\pi \in \mathcal{P}_{p,V}(X)$ and for any $\mu \in \mathcal{P}_{p,V}$,
$\sup\limits_{\|e\|_{Lip} + \|e\|_{V} \leq 1}|\mu P^{n}(e) - \pi(e)| \rightarrow 0$
as $n\rightarrow \infty$. In particular, $\mu P^{n}$ converges weakly to $\pi$.
 \end{prop}
\begin{proof}
The existence of a unique invariant measure $\pi$ is an immediate result of Corollary 4.23 and Theorem 4.25 of \cite{hairernotes}. To show that $\pi \in \mathcal{P}_{p,V}$, reason as follows. If $V$ is a $p$-Lyapunov function, then $V^{p}$ is a $1$-Lyapunov function (for possibly different values of the constants $\beta$ and $K$). Then apply Proposition 4.24 of \cite{hairernotes}.

  We turn to convergence of the expectations $\mu P^{n}(e)$ as $n\rightarrow \infty$. Let $e$ have $\|e\|_{Lip} + \|e\|_{V}< \infty$.  Using 
  (\ref{common-random-numbers-coupling-contraction}) we see 
  $\|Pe\|_{Lip} \leq \alpha\|e\|_{Lip}$ and by iterating the inequality we see 
  \begin{equation}\label{iterate-contraction}
    |P^{n}e(x) - P^{n}e(y)| \leq \alpha^{n}\|e\|_{Lip}d_A(x,y)
  \end{equation}
  By iterating the Lyapunov inequality, we see
  \begin{equation}\label{iterate-lyap}
    |P^{n}e(x) - P^{n}e(y)| \leq \|e\|_{V}\beta^{n}[V(x) + V(y)] + \|e\|_{V}K'
  \end{equation}
  where $K'=2K/(1-\beta)$.
  Combining (\ref{iterate-contraction}) and (\ref{iterate-lyap}), for any coupling $\gamma$ of $\mu$ and $\pi$,
  $$
  |\mu P^{n}(e) - \pi(e)| 
  \leq
  (\|e\|_{Lip}+\|e\|_{V})
  \int_{X\times X}
  \min
  \{\alpha^{n}d_A(x,y),\beta^{n}[V(x)+V(y)] + K'\}
  \,\mathrm{d}\gamma(x,y)
  $$
  It remains to show that right hand side of this inequality tends to $0$ as 
  $n\rightarrow \infty$.
  Letting 
  $f_{n}(x,y)=\min\{\alpha^{n}d_A(x,y),\beta^{n}[V(x)+V(y)] + K'\}$,
  it is clear the pointwise convergence of $f_n$ to $0$ holds. Since also 
  $|f_n| \leq V(x) + V(y) + K'$, the latter function being $\gamma$-integrable,
  the result follows by the dominated convergence theorem.
\end{proof}
Let $x_0$ be an arbitrary basepoint in $X$. 
The next result strengthens the conclusion in case  
$V(x) = 1 + d_{A}(x_0,x)$, 
and concerns contraction in the Wasserstein space 
$\mathcal{P}_{p,A}$. This is the set of all measures that can integrate 
$x \mapsto d_{A}(x_0,x)^p$, together with metric
$$
d_{p,A}(\mu,\nu)
 =
\inf_{\gamma\in\Gamma(\mu,\nu)}
  \left(\int_{X\times X}d_{A}(x,y)^{p}\,\mathrm{d}\gamma(x,y)\right)^{1/p}.
$$
The space $\mathcal{P}_{p,A}$ is complete if $(X,d_A)$ is. 
Furthermore, 
the Kantorovich duality formula holds for $p=1$:
\begin{equation}\label{eqn-df}
\sup_{\|e\|_{Lip}\leq 1}|\mu_1(e) - \mu_2(e)| = d_{1,A}(\mu_1,\mu_2)
\end{equation}
See \cite{villani2008} for more background.
\begin{prop}\label{contraction-strong}
Let the assumptions of Proposition \ref{contraction-ptwise} hold for some $p\geq 1$ and $\alpha < 1$. Let $V(x) = 1 + d_{A}(x,x_0)$ be a $p$-Lyapunov function for the kernel $P$.
Then $P$ determines a contraction mapping on the Wasserstein space $\mathcal{P}_{p,A}(X)$ and possesses a unique invariant measure $\pi \in \mathcal{P}_{p,A}$. Furthermore, if $\mu \in \mathcal{P}_{p,V}$,
\begin{equation}\label{eqn-cr}
\sup\limits_{\|e\|_{Lip} \leq 1}|\mu P^{n}(e) - \pi(e)| \leq \alpha^{n}\sup_{\|e\|_{Lip} \leq 1}|\mu(e) - \pi(e)|.
\end{equation}
 \end{prop}
\begin{proof}
  Let $\gamma$ be any coupling in $\Gamma(\mu_1,\mu_2)$.
  For any points $x,y$ of $X$ we can form a coupling of $\delta_{x}P$ and $\delta_{y} P$ using common random numbers. Formally, this is the measure $C(x,y)$ which arises as the pushforward of $\nu$ under the map $\xi  \mapsto (f(x,\xi),f(y,\xi))$. Then $C$ is a well-defined Markov kernel on $X\times X$, and according to Proposition \ref{contraction-ptwise},
  $$
  \left(\int_{X\times X}d_{A}(x',y')^{p}\,\mathrm{d}(\delta_{(x,y)}C)(x',y')\right)^{1/p} 
  \leq
  \alpha d_{A}(x,y)$$
  Then
  \begin{align*}
    d_{p,A}(\mu_1 P,\mu_2 P)
    &\leq
    \left(\int_{X\times X}d_A(x,y)^{p}\,\mathrm{d}(\gamma C)(x,y)\right)^{1/p} 
    \\&
    \leq
    \alpha \left(\int_{X\times X}d_A(x,y)^{p}\,\mathrm{d}\gamma(x,y)\right)^{1/p}
  \end{align*}
  Since $\gamma$ was arbitrary, it follows that $P$ is a contraction. Since $\mathcal{P}_{p,A}$ is complete, $P$ has a unique stationary measure $\pi$ in $\mathcal{P}_{p,A}$. Inequality (\ref{eqn-cr}) results by combining the contraction property with the duality formula (\ref{eqn-df}).
\end{proof}

Conditions similar to those used in Proposition \ref{contraction-ptwise}
have been mentioned in other works.
The work of \cite{steinsaltz} considered the case of a scalar potential 
$A(x) = V(x)I$. 
The metric viewpoint for the scalar potential 
can be found in 
\cite{hairer2008,stenflo}. The results of \cite{borovkov} may be helpful to find scalar weight functions. The contraction conditions were also motivated by work on contraction analysis for deterministic systems \cite{lohmiller, forni}.

Aside from generality,
there is a reason
related to gradient estimation
for considering matrix-valued functions $A$. 
Even if the underlying system has the unweighted average contraction property,
meaning inequality (\ref{contract-sufficient}) of Proposition \ref{contraction-ptwise} holds 
with the function $A(x)= I$, 
this does not extend to the joint system
(Eqns. \ref{optproc1}, \ref{optproc2}).
This is due to the factor
$m$ in the auxiliary system (\ref{optproc2}),
which makes the Jacobian 
$\frac{\partial T}{\partial z}$
 large at points $(x,m)$ where $\|m\|$ is large.
One approach is to look beyond the scalar potentials
to metrics that weigh the $x$ and $m$ directions differently.
We will see in Section \ref{sect:gradient-estimation} that, for the case of unweighted contraction,
a suitable metric involves a matrix 
$H(x,m)(u_x,u_m) = \Big( (1+h(x,m))u_x,u_m\Big)$
for a scalar function 
$h(x,m)$.

\subsection{Interconnections of contractions}
\label{sect:interconnection-of-contractions}
This section gives conditions for the interconnection of two contracting systems to again be contracting. 
It is relevant to gradient estimation since the system
(\ref{optproc1}, \ref{optproc2})
has a hierarchical form,
the underlying system $x$ feeding into the system $m$. 
Interconnection theorems for contracting systems
hold in other dynamical settings as well;
results for deterministic continuous time systems can be found in \cite{russo, Simpson}. 

Let $X \subseteq \mathbb{R}^{n}$, $Y \subseteq \mathbb{R}^{m}$ be closed, convex sets, and let $Z = X\times Y$. For instance, when these results are applied later to the forward sensitivity process, the space $Y$ will be $L(\mathbb{R}^{n_{X}},\mathbb{R}^{n_{\Theta}})$. Let $(\Xi,\Sigma,\nu)$ be a probability space and let $R$ be the Markov kernel that corresponds to following stochastic recursion on $Z$:
\begin{align*}
x_{n+1} &= f(x_{n},y_{n},\xi_{n+1})\\
y_{n+1} &= g(x_{n},y_{n},\xi_{n+1})
\end{align*}
where the $\xi_{n}$  are independent $\nu$-distributed random variables. For measurable $\phi:Z\to\mathbb{R}$, one has
$(R\phi)(x,y) = \int_{\Xi}\phi( T(x,y,\xi))\,\mathrm{d}\nu(\xi)$ where $T(x,y,\xi) = (f(x,y,\xi),g(x,y,\xi))$. We find conditions on $f$ and $g$ that guarantee the joint system is contracting. 

\begin{asu}\label{asu:two-contractions}
  Regarding the functions $f,g$ and the probability space $(\Xi,\Sigma,\nu)$,
\begin{enumerate}
  \renewcommand{\theenumi}{\roman{enumi}}
  \item 
    The maps $(x,y) \mapsto f(x,y,\xi)$ and $(x,y) \mapsto g(x,y,\xi)$ are $C^{1}$ for each $\xi \in \Xi$,
  \item
    There are pairs $(\|\cdot\|_{X},F)$, $(\|\cdot\|_Y,G)$, such that $\|\cdot\|_{X}, \|\cdot\|_{Y}$ are norms on $\mathbb{R}^{n},\mathbb{R}^{m}$ respectively, $F:X\times Y \to \mathbb{R}^{n\times n}$ and $G:X\times Y \to \mathbb{R}^{m\times m}$ are continuous with values in the invertible matrices, and $\sup_{(x,y) \in X\times Y}\|F(x,y)^{-1}\|_{X} + \|G(x,y)^{-1}\|_{Y} < \infty$,
  \item
    There are $\alpha_1$ and $\alpha_2$, both in $[0,1)$, such that
    \begin{align*}
    \sup_{z \in Z}
    \sup_{u \in \mathbb{R}^{n}:\|u\|_{X}=1}&
    \left(\int_{\Xi}
      \|F(T(z,\xi))\tfrac{\partial f}{\partial x}(z,\xi)F^{-1}(z)u\|_{X}^{p}
      \,\mathrm{d}\nu(\xi)\right)^{1/p}
    \leq
    \alpha_1,
 \\
    \sup_{z \in Z}
    \sup_{u \in \mathbb{R}^{m}:\|u\|_{Y}=1}&
    \left(\int_{\Xi}
      \|G(T(z,\xi))\tfrac{\partial g}{\partial y}(z,\xi)G^{-1}(z)u\|_{Y}^{p}
      \,\mathrm{d}\nu(\xi)\right)^{1/p}
    \leq
    \alpha_2.
\end{align*}
\label{comp-cont-cond}
  \end{enumerate}
\end{asu}
We are concerned with pointwise contraction as in Proposition \ref{contraction-ptwise}. With further integrability assumptions, convergence to a unique stationary measure can be obtained with results of the previous section.
\begin{prop}\label{finsler-hierarchy}
  Let Assumption \ref{asu:two-contractions} hold.
  Let
  $K_1,K_2,$ and $p\geq 1$ be such that
\begin{enumerate}
   \renewcommand{\theenumi}{\roman{enumi}}
\item  $
  \sup\limits_{z\in Z}\sup\limits_{\|u_y\|_{Y} = 1}
  \displaystyle\left(\int_{\Xi}
  \|F(T(z,\xi))\tfrac{\partial f}{\partial y}(z,\xi)G(z)^{-1}u_y\|_{Y}^{p}
  \,\mathrm{d}\nu(\xi)\right)^{1/p}
  \leq K_1
  $,
\item $
  \sup\limits_{z\in Z}\sup\limits_{\|u_x\|_{X} = 1}
  \displaystyle\left(\int_{\Xi}
  \|G(T(z,\xi))\tfrac{\partial g}{\partial x}(z,\xi)F(z)^{-1}u_x\|_{X}^{p}
  \,\mathrm{d}\nu(\xi)\right)^{1/p}
  \leq K_2
  $,
\item $K_1K_2 < (1-\alpha_1)(1-\alpha_2)$.
\end{enumerate}
  Choose $\eta_1,\eta_2$ so that
  $\eta_2K_2 < \eta_1(1-\alpha_1)$
  and 
  $\eta_1K_1 < \eta_2(1-\alpha_2)$.
  Then a pointwise $p$-contraction inequality holds for the system
  $\{ (\Xi,\Sigma,\nu), T, (\|\cdot\|_{Z}, H)\}$ on
  $Z$ where
\begin{subequations}
\begin{flalign}
\label{hierp1}
&H(z)(u_x,u_y) = (F(z)u_x,\, G(z)u_y) \\
\label{hierp2}
&\|(u_x,u_y)\|_{Z} = \eta_1\|u_x\|_{X} + \eta_2\|u_y\|_{Y}
\end{flalign}
\end{subequations}
\end{prop}
\begin{proof}
  We will apply 
  Proposition \ref{contraction-ptwise}.
  We must find an $\alpha <1$ so that
  $$
  \sup_{z\in Z}
  \sup_{u \in \mathbb{R}^{n}\times\mathbb{R}^{m}:\|u\|_Z=1}
  \left(\int_{\Xi}
  \|
  H(T(z,\xi))
  \tfrac{\partial T}{\partial z}(z,\xi)
  H(z)^{-1}
  u
  \|_{Z}^{p}
  \,\mathrm{d}\nu(\xi)\right)^{1/p}
  \leq
  \alpha.
  $$
  Let $z\in Z$ and let $u = (u_x,u_y)$ be any vector with
  $\eta_1\|u_x\|_{X} + \eta_2\|u_y\|_{Y} = 1$. Then
\begin{align*}
  &\left(\int_{\Xi}
  \|
  H(T(z,\xi))\tfrac{\partial T}{\partial z}(z,\xi)H(z)^{-1}
  u
  \|_{Z}^{p}
  \,\mathrm{d}\nu(\xi)\right)^{1/p} \\ 
  &\quad=
  \bigg(\int_{\Xi}\Big[
  \eta_1\|F(T(z,\xi))\tfrac{\partial f}{\partial x}(z,\xi)F(z)^{-1}u_x +
  F(T(z,\xi))\tfrac{\partial f}{\partial y}(z,\xi)G(z)^{-1}u_y\|_{X} \\
  &\quad\quad+
  \eta_2\|G(T(z,\xi))\tfrac{\partial g}{\partial x}(x,\xi)F(z)^{-1}u_x +
  G(T(z,\xi))
  \tfrac{\partial g}{\partial y}(x,\xi)G(z)^{-1}u_y\|_{Y}\Big]^{p}
  \,\mathrm{d}\nu(\xi)\bigg)^{1/p}
  \\&\quad\leq
  \eta_1\alpha_1\|u_x\|_{X} + \eta_1K_1\|u_y\|_{Y} +
  \eta_2 K_2\|u_x\|_{X}  + \eta_2\alpha_2\|u_y\|_{Y} \\
  &\quad
  \leq
  \max
  \left\{
    \alpha_1 + \tfrac{\eta_2}{\eta_1}K_2,
    \alpha_2 + \tfrac{\eta_1}{\eta_2}K_1
  \right\}
\end{align*}
Finally, note that satisfiability of the condition $\max\{\alpha_1 + \frac{\eta_2}{\eta_1}K_2 ,\alpha_2 + \frac{\eta_1}{\eta_2}K_2\} < 1$ is equivalent to the condition $K_{1}K_{2} < (1-\alpha_1)(1-\alpha_2)$.
\end{proof}
The above can be specialized 
to hierarchical interconnections:
\begin{cor}\label{finsler-hierarchy-interconn}
  Let Assumption \ref{asu:two-contractions} hold.
  Say that  $f$ does not depend on $Y$ 
  ($\frac{\partial f}{\partial y} = 0$).
  Let $K$
  be such that
  \begin{equation}\label{bounded-interconnection-derivative}
  \sup_{z\in Z}\sup_{\|u_x\|_{X} = 1}\left(\int_{\Xi}
  \|
  G(T(z,\xi))
  \tfrac{\partial g}{\partial x}(z,\xi)
  F(z)^{-1}
  u_x
  \|_{Y}^{p}
  \,\mathrm{d}\nu(\xi)\right)^{1/p} 
  \leq K.
  \end{equation}
  Choose $\eta_1,\eta_2$ so that
  $\eta_2K < \eta_1(1-\alpha_1)$.
  Then a pointwise $p$-contraction property holds for the system
  $\{ (\Xi,\Sigma,\nu), T, (\|\cdot\|_{Z}, H)\}$
  on
  $Z$ using the $H$ and $\|\cdot\|_{Z}$ of (\ref{hierp1}, \ref{hierp2}).
\end{cor}

The condition (\ref{bounded-interconnection-derivative}) in Corollary \ref{finsler-hierarchy-interconn} can be relaxed using a kind of Lyapunov function for the interconnection of the two systems, while requiring a stronger form of contraction on the input system.
\begin{prop}\label{hierarchy-lyapunov}
  Let Assumption \ref{asu:two-contractions} hold, 
  with $p \geq 2q$ for some $q\geq 1$.
  Let $K$
  and the continuous function 
  $h:Z\to \mathbb{R}_{\geq 0}$ 
  be such that, for all $z \in Z$,
\begin{enumerate}
    \renewcommand{\theenumi}{\roman{enumi}}
    \item $
      \sup\limits_{\|u_x\|_{X}=1}\displaystyle\left(\int_{\Xi}
    \|
    G(T(z,\xi))\tfrac{\partial g}{\partial x}(z,\xi)F^{-1}(z)u_x
    \|_{X}^{q}
    \,\mathrm{d}\nu(\xi)\right)^{1/q} \leq h(z)
  $, \label{lyap-interconn-parti}
\item  $
  \displaystyle\left(
    \int_{\Xi}h(T(z,\xi))^{p}\,\mathrm{d}\nu(\xi)
  \right)^{1/p}  
  \leq h(z) + 
  K
  $. \label{lyap-interconn-partii}
\end{enumerate}
  Then there are some $\eta_1,\eta_2$ so that a pointwise $q$-contraction inequality holds for the system
  $\{(\Xi,\Sigma,\nu), T, (\|\cdot\|_{Z}, H)\}$ on $Z$
  where
\begin{subequations}
\begin{flalign*}
&  H(z)(u_x,u_y) = ((1+\eta_1 h(z))F(z)u_x,\, G(z)u_y)) \\
&  \|(u_x,u_y)\|_{Z} = \|u_x\|_{X} + \eta_2\|u_y\|_{Y}
\end{flalign*}
\end{subequations}
\end{prop}
\begin{proof}
  Let 
  $\alpha_1,\alpha_2$ be contraction coefficients for $f,g$ respectively.  
  Let
  $F_1(z) = [1 + \eta_3h(z)]F(z)$,
  using an $\eta_3 \geq 0$ such that 
  $\alpha_1(1 + \eta_3 K) < 1$.
  We aim to apply
  Corollary \ref{finsler-hierarchy-interconn}
  to the pair of systems $f$ and $g$, using a metric defined by the pairs
  $(\|\cdot\|_{X}, F_1)$
  and $(\|\cdot\|_Y,G)$, in order to find $q$-contraction of the joint system.
  Letting $\|u_x\|_{X}=1$, then,
  \begin{align*}
    &
    \left(\int_{\Xi}
    \|
    F_1(T(z,\xi))
    \tfrac{\partial f}
          {\partial x}(x,\xi)
    F_1(z)^{-1}u_x
    \|_{X}^{q}
    \, \mathrm{d}\nu(\xi)\right)^{1/q} \\
    &\quad\quad=
    \left(\int_{\Xi}
    \left\|
      \frac{1+\eta_3h(T(z,\xi))}
           {1 + \eta_3 h(z)}
           F(T(z,\xi))
           \tfrac{\partial f}
                 {\partial x}(x,\xi)
           F(z)^{-1}u_x
    \right\|_{X}^{q}
    \, \mathrm{d}\nu(\xi)\right)^{1/q}
\end{align*}
Applying H\"olders inequality and the assumption on 
$\tfrac{\partial f}{\partial x}$ yields
\begin{align*}
  &\quad\quad\leq
    \frac{
    1
    }{
      1+\eta_3 h(z)
    }
    \left(1 + \eta_3\left(\int_{\Xi}h(T(z,\xi))^{2q}\, \mathrm{d}\nu(\xi)\right)^{1/(2q)}\right)
    \alpha_1 \\
    &\quad\quad\leq 
    \frac{
      1+\eta_3(h(z) + K)
    }{
      1 +\eta_3 h(z)
    }\alpha_1
        \leq
    \alpha_1(1+\eta_3 )K
\end{align*}
It remains to show that inequality (\ref{bounded-interconnection-derivative}) holds. 
  Let $\|u_x\|_{X}=1$. Then
  \begin{align*}
    &
    \left(
    \int_{\Xi}
    \|G(T(z,\xi))\tfrac{\partial g}{\partial x}(z,\xi)F_1(z)^{-1}u_x\|_{Y}^{q}
    \, \mathrm{d}\nu(\xi)\right)^{1/q}  \\
    &\quad=\frac{1}{1+\eta_3h(z)}
    \left(
      \int_{\Xi}
    \|G(T(z,\xi))\tfrac{\partial g}{\partial x}(z,\xi_2)F(x)^{-1}u_x\|_{Y}^{q}
    \, \mathrm{d}\nu(\xi)\right)^{1/q} 
\leq 
    \frac{h(z)}{1+\eta_3 h(z)} 
    \leq 
    \frac{1}{\eta_3}
    \end{align*}

Let $\eta_1,\eta_2$ be chosen so that
    $
    \eta_2 \frac{1}{\eta_3}
    <
    \eta_1(1-\alpha_1(1 + \eta_3)K)
    $.
    Then by Corollary \ref{finsler-hierarchy-interconn}  the tuple
    $\{ (\Xi,\Sigma,\nu),T, (\|\cdot\|_{Z},H)\}$ 
    determines a $q$-contracting system, where
    $\|(u,v)\|_{Z} = \eta_1\|u\| + \eta_2\|v\|$
    and 
    $H(z)(u_x,u_y) =
    ((1+\eta_3h(z))F(z)u_x , G(z)u_y)$.
    One can take $\eta_1=1$ in these requirements,
    by choosing $\eta_2$ small enough  that
    $
    \eta_2 \frac{1}{\eta_3}
    <
    (1-\alpha_1(1 + \eta_3)K)
    $.
\end{proof}
\section{Stationary differentiability}
\label{sect:stationary-differentiability}
Differentiability of stationary costs is established using properties of the Markov kernel 
$P$. In the next section, the Assumptions are verified based on properties of the derivatives of the system.

Formally differentiating the equation 
$\pi_{\theta} =\pi_{\theta} P_{\theta}$ in $\theta$
suggests the stationary derivative $\pi'$ solves the equation
$l = l P_{\theta} + \pi_{\theta} P_{\theta}'$
in the variable $l$.
By defining $P'$ properly, 
as the linear map 
$e \mapsto \frac{\partial}{\partial \theta}P_{\theta}e$ 
on the space of cost functions, and considering this equation as being between functionals defined on the cost functions, one can show that it has a unique solution $
l^*$, 
which is such that
$l^{*}(e)
=
\frac{\partial}{\partial \theta}\int_{X}e(x)\, \mathrm{d}\pi_{\theta}(x)$. The line of argument used in this section is a variant of Theorem 2 in \cite{heidergott2003taylor}, adapted to the specific ergodicity and state space conditions that we work with. In that work, a class of functions with a norm $\|e\| = \sup_{x}\frac{|e(x)|}{V(x)}$ is considered, while the norm we will use also involves the derivatives of $e$. In the work of \cite{heidergott2003taylor}, an important role is played by the \textit{deviation operator} $D_{\theta}$ (see Section 3 of that work) and that in their setting $D_{\theta}$ maps $\mathcal{E}^2$ back into itself. Dealing directly with the deviation operator in our case requires care since the space of functions will have more subtle topological properties due to the terms involving derivatives. We leave a possible unification of these two approaches to future work.

We introduce the assumptions on $P$ and the cost functions $\mathcal{E}$:
\begin{asu}\label{state-space}
   $X$ is a Polish space,
   $\mathcal{E}$ a vector space of real-valued functions on $X$ with norm 
   $\|\cdot\|_{\mathcal{E}}$,
   and $\mathcal{P}$ a space of probability measures on $X$.
   For any $\mu \in \mathcal{P}$,  it is required that 
   $\sup_{\|e\|_{\mathcal{E}} \leq 1}|\mu(e)| < \infty$.
\end{asu} 
Denote by 
$\Pi_{\theta}$ 
the Markov kernel 
$\Pi_{\theta}(x,A) = \pi_{\theta}(A)$. The parameter space is an open set
$\Theta \subseteq \mathbb{R}^{n_{\Theta}}$ 
and we fix a 
$\theta_0 \in \Theta$.
The space
$\mathbb{R}^{n_{\Theta}}$
has a norm $\|\cdot\|_{\Theta}$.
We show that the map sending a cost function
$e$
to its stationary derivative at the fixed parameter
$\theta_0$
is an element of 
the set $\mathcal{L}$
 of linear maps
from $\mathcal{E}$ to
$L(\mathbb{R}^{n_{\Theta}},\mathbb{R})$ that vanish on the constant functions and are bounded with respect to 
the norm
$\|l\|_{\mathcal{L}} = \sup_{\|e\|_{\mathcal{E}}\leq 1}\|l(e)\|_{\Theta}$:
\begin{equation*}
\mathcal{L} = 
  \{ 
     l \in L(\mathcal{E}, L(\mathbb{R}^{n_{\Theta}},\mathbb{R})) \mid 
     \|l\|_{\mathcal{L}} < \infty,\, l(\mathbf{1}) = 0 
  \}
\end{equation*}
where $\mathbf{1}$ refers to the constant function $x\mapsto 1$. 
Note that $\mathcal{L}$ is a complete space.

To discuss stationary differentiability we introduce the operator 
$\frac{\partial }{\partial \theta}P_{\theta_0}$.
If 
$e \in \mathcal{E}$ 
then
$\frac{\partial}{\partial \theta}P_{\theta_0}e$
is the function from $X$ into 
$L(\mathbb{R}^{n_{\theta}},\mathbb{R})$ 
defined by
$(\frac{\partial}{\partial \theta}P_{\theta_0}e)(x) = \frac{\partial}{\partial\theta}(P_{\theta_0}e(x))$. 

\begin{asu}\label{asu:firstorder}\label{asu:p}
For any  $\theta \in \Theta$ the following hold.
\begin{enumerate}
    \renewcommand{\theenumi}{\roman{enumi}}
\item 
  If $\mu \in \mathcal{P}$ then $\mu P_{\theta} \in \mathcal{P}$ and $P_{\theta}$ has a stationary measure $\pi_{\theta}$ in $\mathcal{P}$, \label{p-is-inv}
\item
  If $e \in \mathcal{E}$ then $P_{\theta}e \in \mathcal{E}$, $\|P_{\theta}\|_{\mathcal{E}} < \infty$,  and 
  $\sum\limits_{i=0}^{\infty}
  \|P_{\theta_0}^{i} - \Pi_{\theta_0}\|_{\mathcal{E}} \leq K_{\theta_0}$ for some $K_{\theta_0} \geq 0$,\label{deviation-bound}
\item 
  For $e\in\mathcal{E}$ and $x\in X$
  the function
  $\theta \mapsto P_{\theta}e(x)$
  is differentiable at $\theta_0$ and
  \(
    \|
    \pi_{\theta_0}\tfrac{\partial}{\partial \theta}P_{\theta_0}
    \|_{\mathcal{L}}
    <
    \infty
  \),
  \label{pi-d-theta-is-l-bounded-operator}
\item 
  \(
  \tfrac{1}{\|\Delta\theta\|_{\Theta}}
  \|
  \pi_{\theta_0}[P_{\theta_0+\Delta\theta} 
      - P_{\theta_0}
      - \tfrac{\partial}{\partial \theta}P_{\theta_0}(\Delta\theta)]
  \|_{\mathcal{E}} 
  \rightarrow 0
  \) 
  as 
  $\|\Delta\theta\|_{\Theta} \rightarrow 0$,
  \label{p-theta-differentiable-at-pi}
\item 
  $\frac{1}{\|\Delta\theta\|_{\Theta}}\|(\pi_{\theta_0+\Delta\theta} - \pi_{\theta_0})[P_{\theta_0+\Delta\Theta} - P_{\theta_0}]\|_{\mathcal{E}} 
  \rightarrow 
  0$ 
  as 
  $\|\Delta\theta\|_{\Theta} \rightarrow 0$.
  \label{dp-theta-continuous-at-pi}
\end{enumerate} \end{asu}

In part \ref{p-theta-differentiable-at-pi}, the functional $\pi_{\theta_0}[ P_{\theta_0 + \Delta\theta} - P_{\theta_0} -\frac{\partial}{\partial \theta}P_{\theta_0}(\Delta \theta)]$ maps a function $e \in \mathcal{E}$ to the number
$
\pi_{\theta_0} P_{\theta + \Delta \theta}(e) -
\pi_{\theta_0} P_{\theta_0}(e) - 
\pi_{\theta_0}(\frac{\partial}{\partial \theta}P_{\theta_0}e(\Delta \theta))$.

The main theorem on stationary differentiability is as follows:
\begin{thm}\label{diff-thm}
  Under Assumptions \ref{state-space} and \ref{asu:p}
  if $e \in \mathcal{E}^{2}$ then 
  $\pi_{\theta}(e)$ is differentiable at $\theta_0$
  and 
  $\tfrac{\partial}{\partial \theta}
   \int_{X}e(x)\,\, \mathrm{d}\pi_{\theta_0}(x)
   = 
   l^{*}(e)$
   where  
  $l^* \in \mathcal{L}$ 
  satisfies 
  $l^* = l^*P_{\theta_0} + \pi_{\theta_0}\frac{\partial}{\partial \theta}P_{\theta_0}$.
\end{thm}
\begin{proof}[Proof of Theorem \ref{diff-thm}]
  First, define 
  $T:\mathcal{L} \to \mathcal{L}$ as
  $
    T(l) := 
    lP_{\theta_0} + 
    \pi_{\theta_0}
      \tfrac{\partial}{\partial \theta}P_{\theta_0}.
$
  That
  $\pi_{\theta_0}
  \frac{\partial}{\partial \theta}
  P_{\theta_0}$
  is in
  $\mathcal{L}$ 
  was one of our assumptions along with $\|P_{\theta}\|_{\mathcal{E}} < \infty$, which implies $T$ is well-defined.
  Let $l^{*}$ be the functional
  $l^{*} 
  =
  \sum\limits_{i=0}^{\infty}(\pi_{\theta_0}\tfrac{\partial}{\partial \theta}P_{\theta_0})P_{\theta_0}^{i}$.
  This is in $\mathcal{L}$ since that space is Banach and  by Part \ref{deviation-bound} of Assumption \ref{asu:firstorder},
  \begin{align*}
    \sum\limits_{i=0}^{\infty}
    \|(\pi_{\theta_0}\tfrac{\partial}{\partial \theta}P_{\theta_0})P_{\theta_0}^{i}\|_{\mathcal{L}}
    &=
      \sum\limits_{i=0}^{\infty}
      \|
      (\pi_{\theta_0}\tfrac{\partial}{\partial \theta}P_{\theta_0})(P_{\theta_0}^{i}- \Pi_{\theta_0})
      \|_{\mathcal{L}} 
      \leq
      \|\pi_{\theta_0}\tfrac{\partial}{\partial \theta}P_{\theta_0}\|_{\mathcal{L}}K.
\end{align*}
To see that $l^{*}$ is  a fixed-point of $T$, note that
$
T(l^{*}) = 
  \sum\limits_{i=1}^{\infty}(\pi_{\theta_0}\tfrac{\partial}{\partial \theta}P_{\theta_0})P_{\theta_0}^{i} + 
  \pi_{\theta_0}\tfrac{\partial}{\partial \theta}P_{\theta_0} = l^{*}.$

To show $l^{*}$ is the unique fixed-point,
let $l$ be any other fixed-point of $T$. Then
$$\|l - l^{*}\|_{\mathcal{L}} = \|T^{n}(l) - T^{n}(l^{*})\|_{\mathcal{L}} 
=
\|(l - l^{*})(P_{\theta_0}^{n}-\Pi_{\theta_0})\|_{\mathcal{L}} \leq 
\|l - l^{*}\|_{\mathcal{L}}\|P_{\theta_0}^{n}-\Pi_{\theta_0}\|_{\mathcal{E}}.$$
Using Part \ref{deviation-bound} of Assumption \ref{asu:firstorder} again, the right hand side of this inequality goes to zero as 
$n\rightarrow \infty$, 
hence $T$ possesses a unique fixed-point 
$l^{*}$ in $\mathcal{L}$.

Define $c(\Delta\theta)$ as the functional
$c(\Delta\theta)(e) = 
\pi_{\theta_0 + \Delta \theta}(e)
- 
\pi_{\theta_0}(e)
-
 l^{*}(e)(\Delta \theta).$
Assumption \ref{state-space} and the definition of $\mathcal{L}$ guarantees $c(\Delta\theta) \in L(\mathcal{E},\mathbb{R})$.
It suffices that 
$
\tfrac{1}{\|\Delta\theta\|_{\Theta}}\|c(\Delta\theta)\|_{\mathcal{E}}
\rightarrow 0
$ 
as
$\Delta\theta \rightarrow 0$.
Using the fact that $T(l^{*}) = l^{*}$, we have
\begin{align*}
c(\Delta\theta)
 &=
   \pi_{\theta_0}
   [P_{\theta_0+\Delta\theta} - P_{\theta_0} - \tfrac{\partial}{\partial\theta}P_{\theta_0}(\Delta \theta)] 
   +
   (\pi_{\theta_0+\Delta\theta} - \pi_{\theta_0})[P_{\theta_0 + \Delta\theta} - P_{\theta_0}]
+c(\Delta\theta)P_{\theta_0}
\end{align*}
Iterating this, and noting that each summand is a functional vanishing on the constant functions,  we obtain that for any $k>0$,
\begin{align*}
  c(\Delta\theta)
 &=
   \pi_{\theta_0}
   (P_{\theta_0+\Delta\theta} - P_{\theta_0} - \tfrac{\partial}{\partial\theta}P_{\theta_0} (\Delta\theta) )\sum\limits_{i=0}^{k-1}(P_{\theta_0}^{i}-\Pi_{\theta_0})\\
 &+
   (\pi_{\theta_0+\Delta\theta} - \pi_{\theta_0})[P_{\theta_0 + \Delta\theta} - P_{\theta_0}]\sum\limits_{i=0}^{k-1}(P_{\theta_0}^{i}-\Pi_{\theta_0}) \\
&+c(\Delta\theta)(P^{k}_{\theta_0}-\Pi_{\theta_0})
\end{align*}
Taking norms and letting $k\rightarrow \infty$, we see that 
\begin{equation*}
\begin{split}
  &\|c(\Delta\theta)\|_{\mathcal{E}} \\
  &\quad\quad\leq
  \|
  \pi_{\theta}
  (P_{\theta_0+\Delta\theta} 
  - 
  P_{\theta_0} 
  -
  \tfrac{\partial}{\partial \theta}P_{\theta_0}(\Delta\theta))
  \|_{\mathcal{E}}K_{\theta_0}
  +
   \|(\pi_{\theta_0+\Delta\theta} - \pi_{\theta_0})[P_{\theta_0 + \Delta\theta} - P_{\theta_0}]\|_{\mathcal{E}}K_{\theta_0}.
\end{split}
\end{equation*}
Finally, use Parts \ref{p-theta-differentiable-at-pi} and \ref{dp-theta-continuous-at-pi} of Assumption \ref{asu:p} 
\end{proof}
\section{State space conditions}
\label{sect:state-space-conditions}
Let 
$P_{\theta}$
be the transition kernel of the Markov chain
\begin{equation}\label{parameterized-state-rep}
x_{n+1} = f(x_{n},\xi_{n+1},\theta)
\end{equation}
with $\nu$-distributed random input $\xi_n$.
In this section we show how Assumptions \ref{asu:lipVA}, \ref{asu:differentiable},  \ref{asu:lipVB},
  and \ref{asu:f} imply
Assumptions \ref{state-space} and \ref{asu:p},
thereby establishing differentiability of the stationary costs for those cost functions $e\in\mathcal{E}^{2}$.

\begin{thm}\label{p-bds}
  Let Assumptions \ref{asu:lipVA} -  
   \ref{asu:f} be satisfied.
  Then Assumptions  \ref{state-space} and
  \ref{asu:p} are verified
  for the space $\mathcal{P}_{2,A}(X)$ of probability measures and the space of cost functions $\mathcal{E}^{2}$, at any $\theta_0 \in \Theta$.
  Hence $\pi_{\theta_0}(e)$ is differentiable
  for any $\theta_0\in \Theta$ and $e \in \mathcal{E}^{2}$.
\end{thm}
To show this, several preliminary results will be used. 
The first is concerned with how 
$P_{\theta}$ varies with $\theta$.  Recall that $x_0$ denotes an arbitrary basepoint.
\begin{prop}
  \label{prop:sufficient-lipschitz-parameter}
  Let 
  $P_{\theta}$
  be the transition kernel of the recursion 
  (\ref{parameterized-state-rep}), where
  \begin{enumerate}
      \renewcommand{\theenumi}{\roman{enumi}}
  \item The map
    $\xi \mapsto d_A(x_0,f(x,\xi,\theta))^{p}$
    is $\nu$-integrable for each 
    $(x,\theta) \in X\times\Theta$,
  \item The function $(x,\theta) \mapsto f(x,\xi,\theta)$ is $C^{1}$ for each $\xi \in \Xi$,
  \item $
  \sup\limits_{(x,\theta)\in X\times\Theta}\sup\limits_{\|u_{\theta}\|=1}
  \left(
    \displaystyle\int_{\Xi}
    \|A(f(x,\xi,\theta))
    \tfrac{\partial f}{\partial \theta}(x,\xi,\theta)
    B(x)^{-1}
    u_{\theta}\|^{p}
    \, \mathrm{d}\nu(\xi)
  \right)^{1/p} \leq K$.
  \end{enumerate}
  Fix a $\theta_0 \in \Theta$.
  Then for all $\Delta\theta$ sufficiently small and all
  $\mu \in \mathcal{P}_{p,A}(X)$
  the inequality
  $
  d_{p,A}(\mu P_{\theta_0},\mu P_{\theta_0+\Delta\theta})
  \leq
  K \|B \Delta\theta \|_{L^{p}(\mu)}
  $ holds.
\end{prop}
\begin{proof}
  Let 
  $\Delta\theta$
  be so small that 
  $\theta_0+t\Delta\theta \in \Theta$ for $t\in[0,1]$.
  If $(x,\xi)$ is distributed according to $\mu \times \nu$ then
  the law of 
  $(f(x,\xi, \theta_0), f(x,\xi,\theta_0 + \Delta\theta))$
  is a coupling of $\mu P_{\theta_0}$ and $\mu P_{\theta_0 + \Delta\theta}$.
  Let
  $\gamma :[0,1] \to \mathbb{R}^{n_\Theta}$
be 
$\gamma(t) = \theta_0 + t\Delta\theta$.
Then
$t \mapsto f(x,\xi,\gamma(t))$,
determines a curve from
$f(x,\xi,\theta_0)$
to
$f(x,\xi,\theta_0+\Delta\theta)$, and reasoning as in Proposition \ref{contraction-ptwise},
\begin{align*}
  &\left(\int_{X}
  \int_{\Xi}
  d_A(f(x,\xi,\theta_0),f(x,\xi,\theta_0+\Delta\theta))^{p}
  \, \mathrm{d}\nu(\xi)\, \mathrm{d} \mu(x)
\right)^{1/p}
\\&\quad\quad\leq
  \left(\int_{X}
  \int_{\Xi}
  \left(
    \int_{0}^{1}
    \|
    A(f(x,\xi,\gamma(t))
    \tfrac{\partial f}{\partial \theta}
    (x,\xi,\gamma(t))
    \Delta\theta
    \|
    \, \mathrm{d}t\right)^{p}\, \mathrm{d}\nu(\xi)\, \mathrm{d}\mu(x)
\right)^{1/p}
\\&\quad\quad\leq
  \left(
    \int_{0}^{1}
    \int_{X}
    \int_{\Xi}
    \|
    A(f(x,\xi,\gamma(t)))
    \tfrac{\partial f}{\partial \theta}
    (x,\xi,\gamma(t))
    \Delta\theta
    \|^{p}\, \mathrm{d}\nu(\xi)\,\mathrm{d}\mu(x)
    \, \mathrm{d}t
\right)^{1/p} \\
&\quad\quad\leq 
\left(\int_{0}^{1}\int_{X}
       K^{p} \|B(x) \Delta\theta \|^{p} \, \mathrm{d} t
\right)^{1/p}
  =
  K\|B \Delta\theta\|_{L^{p}(\mu)}
\end{align*}
\end{proof}

The continuity assumptions on the $L_{X^{i},\Theta^{j}}$ 
ensure 
that integration and differentiation can be exchanged. 
For discussing the differentiability it will be useful to introduce the following concept. 
A function 
$f:X\times\Xi\to \mathbb{R}^{n}$
 is said to be \textit{L$^{1}(\nu)$-continuous} when 
\begin{enumerate}
   \renewcommand{\theenumi}{\roman{enumi}}
\item $x \mapsto f(x,\xi)$ is continuous for each $\xi\in\Xi$,
\item $\xi \mapsto f(x,\xi)$ is measurable for each $x \in X$,
\item $x \mapsto \int_{\Xi}\|f(x,\xi)\|\, \mathrm{d}\nu(\xi)$ is continuous. 
\end{enumerate}
The following two properties are not difficult to show.
(i) If $f,g$ are $L^1(\nu)$-continuous functions 
then so are $\alpha f + \beta g$ for any numbers $\alpha,\beta$.
(ii) A monotonicity property holds:
If 
$f$
is a function satisfying the first two requirements of 
$L^{1}(\nu)$-continuity and
if $\|f(x,\xi)\| \leq \|g(x,\xi)\|$ for an $L^{1}(\nu)$-continuous function $g$, then $f$ is $L^1(\nu)$-continuous.

Using this notion we state a condition for interchanging derivatives and integrals which is a generalized form of a result from \cite{pflug-book}, that considers a scalar parameter.
\begin{thm}[\cite{pflug-book}, Theorem 3.13]\label{thm:pflug-interchange}
  Let 
  $(\Xi,\Sigma,\nu)$
  be a probability space and
  $W \subseteq \mathbb{R}^{n}$
  be an open set. 
  Let 
  $h:W\times\Xi \to \mathbb{R}^{m}$
  be a function such that
\begin{enumerate}\renewcommand{\theenumi}{\roman{enumi}}
 \item 
   $\xi \mapsto h(w,\xi)$
   is integrable for each $w \in W$,
\item 
  $w \mapsto h(w,\xi)$ 
  is continuously differentiable for each $\xi \in \Xi$,
\item
  $\frac{\partial h}{\partial w}$
  is $L^{1}(\nu)$-continuous.
\end{enumerate}
  Then
  $
  \frac{\partial}{\partial w}\int_{\Xi}h(w,\xi)\, \mathrm{d}\nu(\xi) = 
  \int_{\Xi}\frac{\partial h}{\partial w}(w,\xi)\, \mathrm{d}\nu(\xi)
  $ for all $w\in W$.
\end{thm}
This criteria has the useful property that once it is established for 
$f$
it is easily extended to the function 
$e \circ f$.
This is shown in the next proposition.
\begin{prop}\label{prop:interchange-e2}
  Let Assumptions 
  \ref{asu:lipVA}, 
  \ref{asu:differentiable},
  \ref{asu:lipVB} and
  \ref{asu:f} hold.
  If
  $e \in \mathcal{E}^{2}$
  and
  $i+j \leq 2$ 
  then, for any $(x,\theta)\in X\times\Theta$,
  $
  \tfrac{\partial^{i+j}}{\partial x^{i}\partial \theta^{j}}
  \int_\Xi e(f(x,\xi,\theta))\, \mathrm{d}\nu(\xi) = 
  \int_\Xi
  \tfrac{\partial^{i+j}}{\partial x^{i}\partial \theta^{j}}e(f(x,\xi,\theta))
  \, \mathrm{d}\nu(\xi)$.
\end{prop}
\begin{proof}
  Consider the derivative 
  $\frac{\partial}{\partial x}$. To apply Theorem \ref{thm:pflug-interchange},
  it is shown that the map 
  $x \mapsto 
  \int_{\Xi}
  \|
  \tfrac{\partial e}{\partial x}(f(x,\xi,\theta))
  \tfrac{\partial f}{\partial x}(x,\xi,\theta)
  \|
  \, \mathrm{d}\nu(\xi)
  $ is continuous.
  Noting that 
  $$
  \|
  \tfrac{\partial e}{\partial x}(f(x,\xi,\theta))
  \tfrac{\partial f}{\partial x}(x,\xi,\theta)
  \|
  \leq
  \|\tfrac{\partial e}{\partial x}\|_{A}
  \|A(f(x,\xi,\theta))\tfrac{\partial f}{\partial x}(x,\xi,\theta)A(x)^{-1}\|
  \|A(x)\|
  $$
  the result follows by assumption on 
  $\frac{\partial f}{\partial x}$ 
  and the monotonicity property of $L^{1}(\nu)$-continuity.
    Next, consider
  $\frac{\partial^{2}}{\partial \theta^{2}}$.
  We have that
  \begin{align*}
    \|\tfrac{\partial^{2}}{\partial \theta^{2}}e(f(x,\xi,\theta))\|
    &\leq
      \|\tfrac{\partial^{2} e}{\partial x^{2}}\|_{A,A}
      \|
      A(f(x,\xi,\theta))
      \tfrac{\partial f}{\partial \theta}(x,\xi,\theta)
      B(x)^{-1}\|^{2}
      \|B(x)\|^{2}\\
    &\quad+
      \|\tfrac{\partial e}{\partial  x}\|_{A}
      \|
      A(f(x,\xi,\theta))
      \tfrac{\partial^{2} f}{\partial \theta^{2}}(x,\xi,\theta)
      (B(x)^{-1}\oplus B(x)^{-1})
      \|\|B(x)\|^{2}
  \end{align*}
  The $L^{1}(\nu)$-continuity of the left side follows
  by the $L^{1}(\nu)$-continuity of the right side together with the monotonicity property.  Similar reasoning yields the other cases.
\end{proof}
Using this result,
we can obtain the contraction property of $P$
with respect to the class
$\mathcal{E}^{2}$,
and find some bounds on the second order derivatives of $P_{\theta}e$:
\begin{prop}\label{contract-e2-bds}
  Let Assumptions
  \ref{asu:differentiable} - \ref{asu:f}
  be in effect.
  For 
  $e \in\mathcal{E}^{2}$ and $\theta\in\Theta$,
\begin{enumerate}
  \renewcommand{\theenumi}{\roman{enumi}}
\item\label{dx2p-bd} 
  $\|\tfrac{\partial^{2}}{\partial x^{2}}P_{\theta}e\|_{A,A}
  \leq
  K_{X^{2}}\|\tfrac{\partial e}{\partial x}\|_{A} 
  +
  K_{X}^{2}\|\tfrac{\partial^{2}e}{\partial x^{2}}\|_{A,A}$,
\item\label{dt2p-bd} 
  $\|\frac{\partial^{2}}{\partial \theta^{2}}P_{\theta}e\|_{B,B}
  \leq
  K_{\Theta^{2}}\|\tfrac{\partial e}{\partial x}\|_{A} 
  +
  K_{\Theta}^{2}\|\tfrac{\partial^{2}e}{\partial x^{2}}\|_{A,A}$,
\item\label{dxdtp-bd} 
  $\|\frac{\partial^{2}}{\partial x\partial \theta}P_{\theta}e\|_{A,B} 
  \leq
  K_{X,\Theta}\|\tfrac{\partial e}{\partial x}\|_{A} + K_XK_{\Theta}\|\tfrac{\partial^{2}e}{\partial x^{2}}\|_{A,A}$.
  \end{enumerate}      
  Furthermore, for each $\theta$ there is an $L_{\theta} \geq 0$ such that 
  $\|P_{\theta}e\|_{\mathcal{E}^{2}} \leq L_{\theta} \|e\|_{\mathcal{E}^{2}}$
  for all $e \in \mathcal{E}^{2}$.
\end{prop}
\begin{proof}
  We show Part \ref{dt2p-bd};
  Parts \ref{dx2p-bd} and \ref{dxdtp-bd} are established similarly.
  We have
  \begin{align*}
    \tfrac{\partial^{2}}{\partial \theta^{2}}P_{\theta}e
(x)\left(B^{-1}(x) \oplus B^{-1}(x)\right)
  &= T_1 + T_2
\end{align*}
where  $T_1$ and $T_2$ are defined as
$$T_1 = 
\int_{\Xi}
\tfrac{\partial e}{\partial x}(f(x,\xi,\theta))
\tfrac{\partial^{2} f}{\partial \theta^{2}}(x,\xi,\theta)
\left(B(x)^{-1} \oplus B(x)^{-1}\right)
\, \mathrm{d}\nu(\xi),$$
$$T_2 = 
\int_{\Xi}
\tfrac{\partial^{2} e}{\partial x^{2}}(f(x,\xi,\theta))
\left(
  \tfrac{\partial f}{\partial \theta}(x,\xi,\theta)
  B^{-1}(x)
  \oplus 
  \tfrac{\partial f}{\partial \theta}(x,\xi,\theta)
  B^{-1}(x)\right)
\, \mathrm{d}\nu(\xi).$$
Using the identity 
$
A(f(x,\xi,\theta))^{-1}A(f(x,\xi,\theta))
\tfrac{\partial^{2}f}{\partial \theta^{2}}(x,\xi,\theta)
=
\tfrac{\partial^{2}f}{\partial \theta^{2}}(x,\xi,\theta)$,
we get
\begin{align}\label{t1-ineq}
  \|T_{1}\|
&\leq \|\tfrac{\partial e}{\partial x}\|_{A}K_{\Theta^{2}} 
\end{align}
while for $T_2$, use that 
$
 A(f(x,\xi,\theta))^{-1}
 A(f(x,\xi,\theta))
 \tfrac{\partial f}{\partial \theta}(x,\xi,\theta) 
 =
 \tfrac{\partial f}{\partial \theta}(x,\xi,\theta)
$ to get
\begin{align*}
  \|T_2\| 
  &\leq
    \|\tfrac{\partial^{2}e}{\partial x^{2}}\|_{A,A}
    \left(\int_{\Xi}
    \|
    A(f(x,\xi))
    \tfrac{\partial f}{\partial \theta}(x,\xi)
    B^{-1}(x)
    \|^{2}
    \, \mathrm{d}\nu(\xi)\right)
  \leq
    \|
    \tfrac{\partial^{2}e}{\partial x^{2}}\|_{A,A}
    K_{\Theta}^{2}.
\end{align*}
Combining this last inequality with inequality (\ref{t1-ineq}), then,
$$
\|\tfrac{\partial^{2}}{\partial \theta^{2}}P_{\theta}e(x)\|_{B(x),B(x)}
\leq
     \|\tfrac{\partial e}{\partial x}\|_{A}K_{\Theta^{2}} + \|\tfrac{\partial^{2}e}{\partial x^{2}}\|_{A,A}K_{\Theta}^{2}$$
To show the boundedness with respect to $\|\cdot\|_{\mathcal{E}^{2}}$, note that for any 
$e\in\mathcal{E}^{2}$,
\begin{align*}
|(P_{\theta}e)(x)| &\leq |e(x_0)| + \|\tfrac{\partial e}{\partial x}\|_{A}\int_{X}d_{A}(x_0,y)\, \mathrm{d}(\delta_{x}P_{\theta})(y) \\
                 &\leq |e(x_0)| + \|\tfrac{\partial e}{\partial x}\|_{A}[C_{\theta} + K_X d_{A}(x,x_0)] 
\end{align*}
where $C_{\theta}$ is the number $C_{\theta} = \int_{X}d_{A}(x_0,y)\,\mathrm{d}(\delta_{x_0}P_{\theta})(y)$.
This follows, since for the Lipschitz function $h(x) = d(x_0,x)$,
$|(Ph)(x)| \leq |Ph(x_0)| + |(Ph)(x_0) - (Ph)(x)| \leq C_{\theta} + K_Xd_A(x_0,x)$.
Also, for any $x\in X$,
$\frac{|e(x_0)|}{ 1 + d_{A}(x_0,x) } \leq  \frac{|e(x_0)|}{1 + d_{A}(x_0,x_0)}\leq  \|e\|_{A}$.
Therefore
$\|P_{\theta}e\|_{A} \leq \|e\|_{A} + \max\{C_{\theta},K_X\}\|\tfrac{\partial e}{\partial x}\|_{A}.$
\end{proof}
The following quadratic bound
involving the metric $d_A$ will be used as well.
\begin{prop}\label{quadratic-estimate}
Let 
$h:X\to\mathbb{R}^{n}$
be differentiable, such that
$\|\frac{\partial h}{\partial x}(x)A(x)^{-1}\| \leq B(x)$ 
where
$B : X\to\mathbb{R}$ is Lipschitz for the metric $d_A$.
Then the following inequalities hold:
\begin{enumerate}
   \renewcommand{\theenumi}{\roman{enumi}}
\item
  $\|h(x) - h(y)\|
  \leq
  B(x)d_A(x,y) + \tfrac{1}{2}\|B\|_{Lip}d_A(x,y)^{2}$\label{quadratic-pointwise}
\item For any $\mu_1,\mu_2 \in \mathcal{P}_{2,A}(X)$,
$$
\hspace{-1em}\left\|
  \int_{X}h(x)\, \mathrm{d}\mu_1(x) - \int_{X}h(y)\, \mathrm{d}\mu_2(y)
\right\|
\leq
\|B\|_{L^{2}(\mu)}d_{2,A}(\mu_1,\mu_2) + \tfrac{1}{2}\|B\|_{Lip}d_{2,A}(\mu_1,\mu_2)^{2}$$\label{quadratic-measure}
\end{enumerate}
\end{prop}
\begin{proof}
See appendix.
\end{proof}
With these tools in hand we can proceed to the proof of Theorem \ref{p-bds}. 
\begin{proof}[Proof of Theorem \ref{p-bds}]
In order to apply Theorem \ref{diff-thm}, we establish the requirements of Assumptions \ref{state-space} and \ref{asu:p}. Assumption \ref{state-space} requires that for any $\mu$ in $\mathcal{P}_{2,A}(X)$, the bound $\sup\limits_{\|e\|_{\mathcal{E}^{2}} \leq 1}|\mu(e)| < \infty$ holds.  Note that
$|e(x_0)| = \frac{|e(x_0)|}{1+d_A(x_0,x_0)} \leq \|e\|_{A}$. Then
\begin{align*}
|\mu(e)| & \leq \int_{X}\left[|e(x_0)| + \|\tfrac{\partial e}{\partial x}\|_{A} d_A(x_0,x)\right]\,\mathrm{d}\mu(x) 
\leq \max\left\{1, \int_{X}d_{A}(x,x_0)\,\mathrm{d}\mu(x)\right\}\|e\|_{\mathcal{E}^{2}}.
\end{align*} 
The integrability part of Assumption \ref{asu:differentiable} and the contraction part of Assumption \ref{asu:f}  allow us to apply Proposition \ref{contraction-strong}. 
Hence $P_{\theta}$ is a contraction on the space $\mathcal{P}_{2,A}(X)$ with contraction coefficient $K_{X}$, and has a unique invariant measure $\pi_{\theta}$ for each $\theta\in \Theta$. Then
 part \ref{p-is-inv} of Assumption \ref{asu:firstorder} holds.
  Proposition \ref{contract-e2-bds} affirms that $P_{\theta}e \in \mathcal{E}^{2}$ if $e\in\mathcal{E}^{2}$, and $P_{\theta}$ is bounded for the norm $\|\cdot\|_{\mathcal{E}^{2}}$. We now establish 
$\|P_{\theta}^{n} -\Pi_{\theta}\|_{\mathcal{E}^{2}} \leq \rho_{\theta} K_{X}^{n}$ for some constant $\rho_{\theta}$. We consider each of the terms in the norm $\|\cdot\|_{\mathcal{E}^{2}}$ .
First, for $e\in\mathcal{E}^{2}$,
\begin{equation}\label{e1-term}
\|P_{\theta}^{n}(e) - \Pi_{\theta}(e)\|_{A} 
\leq 
K^{n}_{X}\|\tfrac{\partial e}{\partial x}\|_{A}\max\{C_{\theta}, 1\}
\end{equation}
To see this, observe that
\begin{align*}
|P_{\theta}^{n}(e) - \Pi_{\theta}(e))(x)| &= 
  |(P_{\theta}^{n}(e)(x) - P_{\theta}^{n}(e)(x_0) + P_{\theta}^{n}(e)(x_0) - \pi_{\theta}(e)| \\
  &\leq 
  K_{X}^{n}\|\tfrac{\partial e}{\partial x}\|_{A}d_{A}(x,x_0) + 
    K_{X}^{n}\|\tfrac{\partial e}{\partial x}\|_{A}C_{\theta} \\
  &\leq 
  K_{X}^{n}\|\tfrac{\partial e}{\partial x}\|_{A}\max\{C_{\theta},1\}(1 + d_{A}(x,x_0))
\end{align*}
where $C_{\theta} = \int_{X}d_{A}(x_0,y)\,\mathrm{d}\pi_{\theta}(y)$.
Next,
\begin{equation}\label{e2-term}
\|\tfrac{\partial}{\partial x}(P_{\theta}^{n}(e) - \Pi_{\theta}(e))\|_{A} 
\leq
K_{X}^{n}\|\tfrac{\partial e}{\partial x}\|_{A}.
\end{equation}
This inequality follows from Proposition \ref{prop:interchange-e2} and Assumption \ref{asu:f}.
Finally, by recursive application of Part \ref{dx2p-bd} of Proposition \ref{contract-e2-bds},
\begin{equation}\label{e3-term}
\|\tfrac{\partial^{2}}{\partial x^{2}}(P_{\theta}^{n}(e) - \Pi_{\theta}(e))\|_{A,A}
\leq
K_{X^{2}}
K_{X}^{n-1}
\tfrac{1}{1-K_{X}}
\|\tfrac{\partial e}{\partial x}\|_{A} + K_{X}^{2n}\|\tfrac{\partial^{2}e}{\partial x^{2}}\|_{A,A}.
\end{equation}
Adding inequalities (\ref{e1-term}), (\ref{e2-term}), and (\ref{e3-term}), one obtains 
\begin{align*}
\|P_{\theta}^{n}(e) - \Pi_{\theta}(e)\|_{\mathcal{E}^{2}} &\leq  
  K_{X}^n\left( \max\{C_{\theta},1\} + 1 + K_{X^2}\tfrac{1}{K_{X}(1-K_{X})}\right)\|\tfrac{\partial e}{\partial x}\|_{A} + K_{X}^{2n}\|\tfrac{\partial^2 e}{\partial x^2}\|_{A,A} \\
&\leq 
K_{X}^n\left( \max\{C_{\theta},1\} + 1 + K_{X^2}\tfrac{1}{K_{X}(1-K_{X})}\right)\left(\|\tfrac{\partial e}{\partial x}\|_{A} + \|\tfrac{\partial^2 e}{\partial x^2}\|_{A,A}\right) \\
&\leq
 K_{X}^{n} \rho_{\theta} \|e\|_{\mathcal{E}^{2}}
\end{align*}
where $\rho_{\theta} = \max\{C_{\theta},1\} + 1 + K_{X^{2}}\frac{1}{K_{X}(1-K_X)}$. In the second inequality we have used the fact that $K_X < 1$.
 Thus item \ref{deviation-bound} of Assumption \ref{asu:firstorder} is satisfied.

  Proposition \ref{prop:interchange-e2} 
  affirms that
  $\theta \mapsto P_{\theta}e(x)$ 
  is differentiable for 
  $e\in\mathcal{E}^{2}$ and $x\in X$. Proceeding as in the proof there, we see that 
  $\|\tfrac{\partial}{\partial \theta}P_{\theta_0}e(x)\|
  \leq
  \|\frac{\partial e}{\partial x}\|_{A}K_{\Theta}\|B(x)\|$.
  Therefore
  $
  \|
  \pi_{\theta_0}\tfrac{\partial}{\partial \theta}P_{\theta_0}
  \|_{\mathcal{L}}
  \leq
  K_{\Theta}\|B\|_{L^{1}(\pi_{\theta_0})}
  $,
  which confirms 
  Part \ref{pi-d-theta-is-l-bounded-operator} of Assumption \ref{asu:p}.
  
  Part \ref{dt2p-bd} of Proposition \ref{contract-e2-bds} means that
  for any $e\in\mathcal{E}^{2}$ and $\theta\in\Theta$,
  $
  \|
  \tfrac{\partial^{2}}{\partial \theta^{2}}P_{\theta}e(x)
  \|_{B(x),B(x)}
  \leq
  k_1\|e\|_{\mathcal{E}^{2}}
  $
  where
  $k_1 
  =
  \max\{K_{\Theta}^{2},K_{\Theta^{2}}\}$.
  Using the 2nd order version of Taylor's theorem,
  this implies that for all $\Delta\theta$ sufficiently small,
  for all
  $e\in\mathcal{E}^{2}$,
  and $x\in X$, we have
  \begin{equation}\label{ptwise-der-eq}
  |
  P_{\theta_0+\Delta\theta}e(x) 
  -
  P_{\theta_0}e(x) 
  -
  \tfrac{\partial}{\partial \theta}P_{\theta_0}e(x)
  (\Delta\theta)
  | 
  \leq
  \tfrac{1}{2}k_1\|e\|_{\mathcal{E}^{2}}\|B(x)\Delta\theta\|^{2}
  \end{equation}
  Integrating inequality (\ref{ptwise-der-eq}) and dividing by 
  $\|\Delta\theta\|$ leads to
  $$
  \tfrac{1}{\|\Delta\theta\|}
  \|\pi_{\theta_0}[
  P_{\theta_0+\Delta\theta}-P_{\theta_0}
  -
  \tfrac{\partial}{\partial \theta}P_{\theta_0}(\Delta\theta)
  ]
  \|_{\mathcal{E}^{2}}
  \leq 
  \tfrac{1}{2}k_1\|B\|_{L^{2}(\pi_{\theta_0})}^{2}\|\Delta\theta\|
  $$
  and the right hand side goes to zero as $\|\Delta\theta\| \rightarrow 0$.
  Only Part \ref{dp-theta-continuous-at-pi} of Assumption \ref{asu:p} remains.
By the fundamental theorem of calculus,
$$(P_{\theta_0 + \Delta\theta} - P_{\theta_0})e(x) = \int_{0}^{1}\int_{\Xi}\tfrac{\partial e}{\partial x}(f(x,\xi,\theta+\lambda\Delta\theta))\tfrac{\partial f}{\partial \theta}(x,\xi,\theta + \lambda\Delta\theta)\Delta\theta \, \mathrm{d}\nu(\xi)\,\mathrm{d}t$$
Differentiating the above with respect to $x$ and using Part \ref{part-other-asu} of Assumption \ref{asu:f} yields
$$
\|
\tfrac{\partial}{\partial x}((P_{\theta_0 + \Delta\theta} - P_{\theta_0})e(x))A(x)^{-1}\| \leq \|e\|_{\mathcal{E}^{2}}k_2\|\Delta\theta\|\|B(x)\|$$
where $k_2 = \max\{K_{X,\Theta},K_{X}K_{\Theta}\}$.
 Applying Proposition \ref{quadratic-estimate} we have
\begin{equation*}
\begin{split}
  &\|(\pi_{\theta_0 + \Delta\theta} - \pi_{\theta})(P_{\theta_0 + \Delta\theta} - P_{\theta_0})e\|\\  &\quad\quad\leq
  k_2\|\Delta\theta\|\|e\|_{\mathcal{E}^{2}}
  \left[
    \|B\|_{L^{2}(\pi_{\theta_0})}d_{2,A}(\pi_{\theta_0+\Delta\theta},\pi_{\theta_0})
    + 
    \tfrac{1}{2}\|B\|_{Lip}d_{2,A}(\pi_{\theta_0 + \Delta\theta},\pi_{\theta_0})^{2}
  \right]
\end{split}
\end{equation*}
For the terms $d_{2,A}$, first apply the contraction property of $P$ and Proposition \ref{prop:sufficient-lipschitz-parameter}:
\begin{align*}
  d_{2,A}(\pi_{\theta+\Delta\theta},\pi_{\theta})
  &\leq
  d_{2,A}(
  \pi_{\theta+\Delta\theta} P_{\theta+\Delta\theta}, \pi_{\theta}P_{\theta+\Delta\theta}
  )
  +
  d_{2,A}(\pi_{\theta}P_{\theta+\Delta\theta},\pi_{\theta}P_{\theta}) \\
  &\leq 
K_X d_{2,A}(\pi_{\theta+\Delta\theta},\pi_{\theta}) 
+ 
K_{\Theta}\|B \Delta\theta \|_{L^{2}(\pi_{\theta})}
\end{align*}
Rearranging terms yields 
$d_{2,A}(\pi_{\theta+\Delta\theta},\pi_{\theta}) 
\leq 
\tfrac{1}{1-K_X}K_{\Theta}\|B \Delta\theta \|_{L^{2}(\pi_{\theta})}$.
Hence
\begin{align*}
&\|(\pi_{\theta_0 + \Delta\theta} - \pi_{\theta})(P_{\theta_0 + \Delta\theta} - P_{\theta_0})
\|_{\mathcal{L}} 
\\&\quad\quad\leq
k_2\|B\|^{2}_{L^{2}(\pi_{\theta_0})}\|\Delta\theta\|\left[
  \tfrac{1}{1-K_X}K_{\Theta}\|\Delta\theta\| 
  +
  \tfrac{1}{2}\|B\|_{Lip}\left(\tfrac{1}{1-K_X}K_{\Theta}\|\Delta\theta \|\right)^{2}\right],
\end{align*}
and Part \ref{dp-theta-continuous-at-pi} of Assumption \ref{asu:p} is verified.
\end{proof}
\section{Gradient estimation}
\label{sect:gradient-estimation}
The goal of this section is to prove Theorem \ref{mainthm}.
The standing assumptions  are
Assumptions 
\ref{asu:lipVA} - \ref{asu:f}.
We let 
$Z=X\times M$
and denote elements of this space by
$z = (x,m)$. 
Denote by $R_{\theta}$ the Markov kernel corresponding to the recursion (\ref{optproc1}, \ref{optproc2}).
In Proposition \ref{joint-metric} and Corollary \ref{ergodicity-sensitivity} we establish convergence of the forward sensitivity system in the sense of Proposition \ref{contraction-weak}. 
It involves finding an appropriate Lyapunov function $V$ and metric $d_{H}$ on $X\times M$.
In Proposition \ref{partial-e-mult-m-lipschitz} we show that 
$(x,m) \mapsto \frac{\partial e}{\partial x}(x)m$
is an integrable function for $\gamma_{\theta}$, thereby establishing that the right  hand side of (\ref{der-comp-id}) is finite.
 Finally, we want to show that the functional $l$ defined by 
\begin{equation}\label{def:l-stationary}
l(e) = 
\int_{X\times M}
\tfrac{\partial e}{\partial x}(x)m
\, \mathrm{d}\gamma_{\theta}(x,m)
\end{equation}
is bounded for the norm $\|\cdot\|_{\mathcal{L}}$ and satisfies the derivative equation of Theorem \ref{diff-thm}.

Define $g$ and $T$ to be the functions
\begin{align}\label{auxilliary-component}
g( (x,m),\xi,\theta) = 
  \tfrac{\partial f}
       {\partial x}(x,\xi,\theta)m + 
       \tfrac{\partial f}
  {\partial \theta}(x,\xi,\theta),
\end{align}
$$
T((x,m),\xi,\theta)
= 
\left(
  f(x,\xi,\theta), 
  g(x,m,\xi,\theta)
\right).
$$
As $\theta$ is fixed in this section, we simplify notation and denote the values of $g$ by $g(z,\xi)$.
We use $u_x, u_{\theta},u_{m}$
to denote vectors in 
$\mathbb{R}^{n_{X}},\mathbb{R}^{n_{\Theta}},$ and 
$L(\mathbb{R}^{n_{\Theta}},\mathbb{R}^{n_{X}}),$ respectively.
\begin{prop}\label{joint-metric}
  Define $h:Z\to \mathbb{R}_{\geq 0}$ as
  $h(z) = \eta_1\|A(x)m\| + \eta_2\|B(x)\| + \eta_3d_{A}(x_0,x).$
  Then there are $\eta_1,\eta_2,\eta_3,\eta_4,\eta_5$ 
  so that 
  $\{ (\Xi,\Sigma,\nu), T, (\|\cdot\|_{Z},H)\}$
  satisfies a $1$-contraction inequality where
  $$H(z)(u_x,u_m) = \Big((1+\eta_4 h(z))A(x)u_x, A(x)u_m\Big),$$
  $$\|(u_x,u_m)\|_{Z} = \|u_x\| + \eta_5\|u_m\|.$$
\end{prop}
\begin{proof}
  We will apply Proposition \ref{hierarchy-lyapunov} to the map 
  $ T(z,\xi) = (f(x,\xi,\theta),g(x,m,\xi))$,
  to find contraction in the metric $d_{H}$.  The norm $\|\cdot\|_{M}$ is the usual norm on $M$ induced by $\|\cdot\|_{X}$ and $\|\cdot\|_{\Theta}$. 
For Part \ref{comp-cont-cond} of Assumption \ref{asu:two-contractions}, we have 
\begin{align*}
&\sup_{\|u_{m}\|=1}
\int_{\Xi}
\|
A(f(x,\xi,\theta))\tfrac{\partial g}{\partial m}(z,\xi)A(x)^{-1}u_{m}
\|
\, \mathrm{d}\nu(\xi) \\
&\quad\quad=
\sup_{\|u_{m}\|=1}
\int_{\Xi}\sup_{\|u_x\|=1}
\|
A(f(x,\xi,\theta))\tfrac{\partial f}{\partial x}(x,\xi,\theta)A(x)^{-1}u_{m}u_x
\|
\, \mathrm{d}\nu(\xi) 
\leq K_{X}
\end{align*}
and, directly by assumption,
\begin{align*}
&\sup_{\|u_x\|=1}
\left(\int_{\Xi}
\|
A(f(x,\xi,\theta))\tfrac{\partial f}{\partial x}(x,\xi,\theta)A(x)^{-1}u_x
\|^{2}
\, \mathrm{d}\nu(\xi)\right)^{1/2} \leq K_{X}.
\end{align*}

We now establish Part \ref{lyap-interconn-parti} of Proposition \ref{hierarchy-lyapunov}. The function 
  $\frac{\partial g}{\partial x}(z,\xi)$
  is a linear map from 
  $\mathbb{R}^{n_{X}}$
  to
  $L(\mathbb{R}^{n_{\Theta}},\mathbb{R}^{n_{X}})$,
  and we identify this with a bilinear map from
  $\mathbb{R}^{n_{X}} \times \mathbb{R}^{n_{\Theta}}$
  to
  $\mathbb{R}^{n_{X}}$.
  Specifically,
  $$
  \tfrac{\partial g}{\partial x}(z,\xi)[u_x,u_{\theta}] =
  \tfrac{\partial^{2} f}{\partial x^{2}}(x,\xi,\theta)[u_x,m\, u_{\theta}] +
  \tfrac{\partial^{2}f}{\partial x\partial \theta}(x,\xi,\theta)[u_x,u_{\theta}]
  $$
  and
  $A(f(x,\xi,\theta))\frac{\partial g}{\partial x}(z,\xi)A(x)^{-1}$
  is the linear map from 
  $\mathbb{R}^{n_{X}}$
  to
  $L(\mathbb{R}^{n_{\Theta}},\mathbb{R}^{n_{X}})$ where
  \begin{align*}
    &
    A(f(x,\xi,\theta))
    \tfrac{\partial g}{\partial x}(z,\xi)
    A(x)^{-1}
    [u_x,u_{\theta}] 
    =\\
    &
    A(f(x,\xi,\theta))
    \tfrac{\partial^{2} f}{\partial x^{2}}(x,\xi,\theta)
    [A(x)^{-1}u_x,m\,u_{\theta}]
    +
    A(f(x,\xi,\theta))
    \tfrac{\partial^{2}f}{\partial x\partial \theta}(x,\xi,\theta)
    [A(x)^{-1}u_x,u_{\theta}]
  \end{align*}

  For the first term we have, using the assumption on 
  $\tfrac{\partial^{2}f}{\partial x^{2}}$ from Assumption \ref{asu:f}
  and the identity
  $m\,u_{\theta} = A(x)^{-1}A(x)m\,u_\theta$,
  $$
  \sup_{\|u_x\|=1}
  \int_{\Xi}\sup_{\|u_{\theta}\|=1}
  \|
  A(f(x,\xi,\theta)
  \tfrac{\partial^{2} f}{\partial x^{2}}(x,\xi,\theta)
  [A(x)^{-1}u_x, m\,u_{\theta}]
  \|
  \, \mathrm{d}\nu(\xi) \leq K_{X^{2}}\|A(x)m\|
  $$
  For the second, use the identity 
  $u_{\theta}= B(x)^{-1}B(x)u_{\theta}$
  and our assumption on 
  $\tfrac{\partial^{2} f}{\partial x\partial \theta}$,
  $$
  \sup_{\|u_x\|=1}
  \int_{\Xi}
  \sup_{\|u_{\theta}\|=1}
  \|
  A(f(x,\xi,\theta))
  \tfrac{\partial^{2} f}{\partial x \partial \theta}(x,\xi,\theta)
  [A(x)^{-1}u_x,u_{\theta}]
  \|
  \, \mathrm{d}\nu(\xi)
  \leq 
  K_{X,\Theta}\|B(x)\|
  $$
  Combining these two inequalities, while assuming $K_{X^{2}} \leq \eta_1$ and $K_{X,\Theta} \leq \eta_2$,
  \begin{align*}
    \sup_{\|u_{x}\|=1}
  \int_{\Xi}
  \|A(f(x,\xi,\theta))\tfrac{\partial g}{\partial x}(z,\xi)A(x)^{-1}u_x\|
  \, \mathrm{d}\nu(\xi) 
  &\leq 
  K_{X^{2}}\|A(x)m\| + K_{X,\Theta}\|B(x)\| \\
  &\leq h(z)
  \end{align*}
Next, we confirm Part \ref{lyap-interconn-partii} of Proposition \ref{hierarchy-lyapunov}, by showing the Lyapunov property of the function $h$. 
We consider the three terms of the function, starting with
 $\|A(x)m\|$:
\begin{align*}
  \left(
  \int_{\Xi}\|A(f(x,\xi,\theta))g(z,\xi)\|^{2}\, \mathrm{d}\nu(\xi)
  \right)^{1/2}
  &\leq
    \left(
    \int_{\Xi}
    \|
    A(f(x,\xi,\theta))
    \tfrac{\partial f}{\partial x}(x,\xi,\theta)m
    \|^{2}
    \, \mathrm{d}\nu(\xi)
    \right)^{1/2} \\
  &+
    \left(
    \int_{\Xi}
    \|A(f(x,\xi,\theta))\tfrac{\partial f}{\partial \theta}(x,\xi,\theta)\|^{2}
    \, \mathrm{d}\nu(\xi)\right)^{1/2} \\
  &\leq 
  K_{X} \|A(x)m\| + K_{\Theta}\|B(x)\|
\end{align*}
Next is $\|B(x)\|$. 
Fix a basepoint $x_0$ and set 
$
B_0
=
\left(\int_{\Xi}\|B(f(x_0,\xi,\theta))\|^{2}\, \mathrm{d}\nu(\xi)\right)^{1/2}
$. Then
\begin{align*}
\left(\int_{\Xi}\|B(f(x,\xi,\theta))\|^{2}\, \mathrm{d}\nu(\xi)\right)^{1/2} 
&\leq
B_0
+
\|B\|_{Lip}
  \left(\int_{\Xi}
    d_A(f(x_0,\xi,\theta),f(x,\xi,\theta))^{2}
    \, \mathrm{d}\nu(\xi)\right)^{1/2} \\
&\leq B_0 + \|B\|_{Lip}\,K_{X} d_A(x_0,x)
\end{align*}
The first inequality uses Assumption \ref{asu:lipVB} and the second uses the pointwise contraction property of $f$
which comes from Proposition \ref{contraction-ptwise}.
For the term 
$d_A(x_0,x)$
we have, setting 
$D_0 
=
\left(\int_{\Xi}d_A(x_0,f(x_0,\xi,\theta))^{2}\, \mathrm{d}\nu(\xi)\right)^{1/2}$,
\begin{align*}
\left(
\int_{\Xi}d_A(x_0,f(x,\xi,\theta))^{2}\, \mathrm{d}\nu(\xi)
\right)^{1/2} &\leq 
D_0 + 
\left(
\int_{\Xi}d_A(f(x_0,\xi,\theta),f(x,\xi,\theta))^{2}\, \mathrm{d}\nu(\xi)
\right)^{1/2} \\
&\leq D_0 + K_{X} d_A(x_0,x)
\end{align*}
Combining these we get
\begin{align*}
  &\left(\int_{\Xi}h(T(z,\xi))^{2}\, \mathrm{d}\nu(\xi)\right)^{1/2} \\
  &\quad\quad \leq 
  \eta_1K_{X} \|A(x)m\| +
  \eta_1 K_{\theta}\|B(x)\| + 
  (\eta_2\|B\|_{Lip}K_{X}  + 
  \eta_3 K_{X} )d_A(x_0,x) + K_4
\end{align*}
where 
$K_4 = 
\eta_2B_0
+
\eta_3D_0$.
Based on this inequality, it is evident that $\eta_1,\eta_2,\eta_3$ can be chosen so that the Lyapunov condition on $h$ is satisfied.
Specifically, take
$
K_{X^{2}} \leq \eta_1$, 
$\max\{K_{X,\Theta},\eta_1K_{\Theta}\} < \eta_2$, and 
$\eta_2\|B\|_{Lip}K_{X} < \eta_3(1-K_{X})$.
\end{proof}
We can use $h$ to get a Lyapunov function, yielding ergodicity of the sensitivity process:
\begin{cor}\label{ergodicity-sensitivity}
  Let the $\eta_1,\eta_2,\eta_3$ of Proposition \ref{joint-metric} be chosen so that they are all positive.
  Let $V$ be the function
  $V(z) = \eta_1\|A(x)m\| + \eta_2 \|B(x)\| + \eta_3d_A(x_0,x) + 1.$
  Then
  the kernel $R_{\theta}$ has a unique invariant measure
  $\gamma_{\theta} \in \mathcal{P}_{1,V}(Z)$, and for $\mu \in \mathcal{P}_{1,V}(Z)$,
  $\sup_{\|g\|_{Lip(H)} + \|g\|_{V}\leq 1}|\mu R_{\theta}^{n}(g) - \gamma_{\theta}(g)| \rightarrow 0$ as $n\rightarrow \infty$.
\end{cor}
\begin{proof}
We apply Proposition \ref{contraction-weak}, using the metric $d_{H}$ defined in Proposition \ref{joint-metric}. Proposition \ref{joint-metric} established the pointwise contraction inequality needed for Proposition \ref{contraction-weak}. For some $\beta \in [0,1)$,  the inequality 
$\int_{\Xi}V(T(z,\xi,\theta))\, \mathrm{d}\nu(\xi) \leq \beta V(z) + (K_{4} + 1)$
 holds at $z \in Z$, as we have already shown in the proof of Proposition \ref{joint-metric}. It remains to show that $V$ has compact sublevel sets.
  Note that if $V(x,m) \leq r$ then  $d_A(x_0,x) \leq \tfrac{r}{\eta_3}$ and $\|m\| \leq \tfrac{r K}{\eta_1}$, where $K$ is such that 
  $\sup_{x\in X}\|A(x)^{-1}\| \leq K$. Thus
  $V^{-1}[0,r]$ is contained in the compact set
  $\{ (x,m) \in Z \mid d_{A}(x_0,x) \leq \frac{r}{\eta_3} \text{ and } \|m\| \leq \frac{rK}{\eta_1} \}$.
\end{proof}
To ensure that the function $(x,m) \mapsto \frac{\partial e}{\partial x}(x)m$ is integrable for the measure $\gamma_{\theta}$ it suffices that it is Lipschitz for the metric $d_H$, and bounded for Lyapunov function $V$:
\begin{prop}\label{partial-e-mult-m-lipschitz}
  For any 
  $e \in \mathcal{E}^{2}$
  the map
  $(x,m) \mapsto \frac{\partial e}{\partial x}(x)m$
  is a Lipschitz function 
  in the metric $d_H$ of Proposition \ref{joint-metric}, and
  is also bounded for the norm $\|\cdot\|_{V}$.
\end{prop}
\begin{proof}
Let the $\eta_i$ be as in Proposition \ref{joint-metric}.
Let $g(x,m) = \frac{\partial e}{\partial x}(x)m$.
We have 
$$\|g(x,m)\| \leq \|\tfrac{\partial e}{\partial x}\|_{A}\|A(x)m\| \leq \|e\|_{\mathcal{E}^{2}}\|A(x)m\| \leq \tfrac{1}{\eta_1}\|e\|_{\mathcal{E}^{2}}V(x,m)$$ hence $\|g\|_{V} \leq \tfrac{1}{\eta_1}\|e\|_{\mathcal{E}^{2}}$.
Next, we show that $\|g\|_{Lip} < \infty$ for the metric $d_{H}$. This is equivalent to showing $\|\frac{\partial g}{\partial x}\|_{H} < \infty$.
Let 
$(u_x,u_m)$
be a vector in 
$\mathbb{R}^{n_{X}} \times L(\mathbb{R}^{n_{\Theta}},\mathbb{R}^{n_X})$.
Then $H(z)^{-1}(u_x,u_m)$ is
$
H(z)^{-1}(u_x,u_m)
=
\left(\tfrac{1}{1+\eta_4h(z)}A^{-1}(x)u_x, A(x)^{-1}u_m\right)
$
and 
$\frac{\partial g}{\partial z}(z)$
is the linear map from 
$\mathbb{R}^{n_{X}} \times L(\mathbb{R}^{n_{\Theta}},\mathbb{R}^{n_{X}})$
to
$L(\mathbb{R}^{n_{\Theta}},\mathbb{R})$
where
\[
\tfrac{\partial g}{\partial z}(z)[u_x,u_m][u_{\theta}] = 
\tfrac{\partial^{2} e}{\partial x^{2}}(x)[u_x,m u_{\theta}] + 
\tfrac{\partial e}{\partial x}(x)[u_mu_{\theta}]
\]
Fix $(u_x,u_m)$ with $\|u_x\| + \eta_5\|u_m\| =1$. Then
\begin{align*}
  &\left\|
    \tfrac{\partial g}{\partial z}(z)H(z)^{-1}(u_x,u_m)
  \right\| 
  \\&\quad\quad=
  \sup_{\|u_{\theta}\|=1}
  \left|
  \frac{
    \tfrac{\partial^{2} e}{\partial x^{2}}(x)[A^{-1}(x)u_x,m u_{\theta}]
  }{
    1+\eta_4h(z)
  } 
  + 
  \tfrac{\partial e}{\partial x}(x)A^{-1}(x) u_m u_{\theta}
  \right|
  \\&\quad\quad\leq
  \sup_{\|u_{\theta}\|=1}
  \frac{
    \|\tfrac{\partial^{2}e}{\partial x^{2}}\|_{A,A}\|u_x\|\|A(x)m\|\|u_{\theta}\|
  }{
    1+\eta_4h(z)
  }
  +
  \|\tfrac{\partial e}{\partial x}\|_{A}\|u_m\|\|u_{\theta}\| \\
  &\quad\quad\leq
  \frac{
    \|\tfrac{\partial^{2}e}{\partial x^{2}}\|_{A,A}\|u_x\|\|A(x)m\|
  }{
    1+\eta_4h(z)
  }
  + 
  \|\tfrac{\partial e}{\partial x}\|_{A}\|u_m\|
\end{align*}
To continue, note by definition of $h$ that 
$\frac{\|A(x)m\|}{1 +\eta_4h(z)} \leq \frac{1}{\eta_1\eta_4}$.
Then, 
\begin{align*}
  &\leq 
    \|\tfrac{\partial^{2}e}{\partial x^{2}}\|_{A,A}\|u_x\|
    \tfrac{1}{\eta_1\eta_4} 
    +
    \tfrac{\eta_5}{\eta_5}
    \|\tfrac{\partial e}{\partial x}\|_{A}\|u_m\| \\
  &\leq 
  \max\left\{
  \|\tfrac{\partial^{2}e}{\partial x^{2}}\|_{A,A}\tfrac{1}{\eta_1\eta_4},
    \tfrac{1}{\eta_5}\|\tfrac{\partial e}{\partial x}\|_{A}
  \right\} \\
  &\leq \|e\|_{\mathcal{E}^{2}}
  \max\left\{
  \tfrac{1}{\eta_1\eta_4},
  \tfrac{1}{\eta_5}
  \right\}
\end{align*}
Therefore a Lipschitz constant for the function $g$ is 
$\|e\|_{\mathcal{E}^{2}}
\max\left\{
\tfrac{1}{\eta_1\eta_4},
\tfrac{1}{\eta_5}
\right\}$.
\end{proof}
We now continue to the proof of Theorem \ref{mainthm}.
\begin{proof}[Proof of Theorem \ref{mainthm}]
By Corollary \ref{ergodicity-sensitivity}, the forward sensitivity process converges to a unique stationary measure $\gamma_{\theta}$ in $\mathcal{P}_{1,V}(Z)$.
Let $g$ be the function $g(x,m) = \frac{\partial e}{\partial x}(x)m$.
By Proposition \ref{partial-e-mult-m-lipschitz} we 
see that $\|g\|_{Lip} + \|g\|_{V} < \infty$, which means in particular that the integral on the right side of equation (\ref{def:l-stationary}) is well-defined.

We show that the functional $l$ of (\ref{def:l-stationary}) is bounded for the norm $\|\cdot\|_{\mathcal{L}}$. 
We have 
$
\|l(e)\|
\leq
\|e\|_{\mathcal{E}^{2}}
\int_{Z}\|A(x)m\|\, \mathrm{d}\gamma_{\theta}(z),$
with the latter integral being finite since $\gamma_{\theta} \in \mathcal{P}_{1,V}(Z)$.
Then $\|l\|_{\mathcal{L}} < \infty$. It remains to show $T(l) = l$.
  By the identity 
  $\gamma_{\theta}=\gamma_{\theta} R_{\theta}$,
  \begin{align}\label{l-id}
    &l(e) = 
    \int_{X \times M}
    \tfrac{\partial e}{\partial x}(x)m \, \mathrm{d}\gamma_{\theta}(x,m) \nonumber \\
    &\, \, \, \, \, \, \, \, \,
    =
    \int_{X\times M}
    \left(
      \int_{\Xi}
      \tfrac{\partial e}{\partial x}(f(x,\xi,\theta))
      \left(
        \tfrac{\partial f}{\partial x}(x,\xi,\theta)m +
        \tfrac{\partial f}{\partial \theta}(x,\xi,\theta)
      \right)
      \, \mathrm{d}\nu(\xi)
    \right)d\gamma_{\theta}(x,m)
  \end{align}  
  Recall the definition of $T$ is 
  $
  T(l)e
  =
  lP_{\theta}e 
  +
  \pi_{\theta}\tfrac{\partial}{\partial \theta}P_{\theta}e
  $. 
  With our definition of $l$, and applying
  Proposition \ref{prop:interchange-e2}, these two terms are
\begin{align}\label{lp-expansion}
  lP_{\theta}(e)&=
\int_{X\times M}
\tfrac{\partial}{\partial x}(P_{\theta}e)(x)m 
\, \mathrm{d}\gamma_{\theta}(x,m) \nonumber\\
&= 
\int_{X\times M}
\left(
  \int_{\Xi}
  \tfrac{\partial e}{\partial x}(f(x,\xi,\theta))
  \tfrac{\partial f}{\partial x}(x,\xi,\theta)\, \mathrm{d}\nu(\xi)
\right)m
\, \mathrm{d}\gamma_{\theta}(x,m),
\end{align}
and
\begin{align}\label{pdp-expansion}
  &\pi_{\theta}\tfrac{\partial}{\partial \theta}P_{\theta}e
  =
  \int_{X}
  \left(
    \int_{\Xi}
    \tfrac{\partial e}{\partial x}(f(x,\xi,\theta))
    \tfrac{\partial f}{\partial \theta}(x,\xi,\theta)
    \, \mathrm{d}\nu(\xi)
  \right)\, \mathrm{d}\pi_{\theta}(x).
\end{align}
Add equation (\ref{lp-expansion}) to (\ref{pdp-expansion}) and compare with  (\ref{l-id}) to see $T(l) = l$.
\end{proof}
To finish this section, let us discuss how this estimator can be implemented. One option is to iterate the joint recursion (\ref{optproc1}, \ref{optproc2}) for a large number of steps, to obtain a sample $(x_{n}, m_{n})$, and then prepare the estimate by forming the product $\Delta_{n} = \frac{\partial e}{\partial x}(x_n)m_{n}$. This requires the ability to compute the derivatives of $e$ and $f$. According to Theorem \ref{mainthm}, the estimate $\Delta_{n}$ has the property that $\mathbb{E}[\Delta_{n}] \rightarrow \frac{\partial }{\partial \theta}\int_{X}e(x)\, \mathrm{d}\pi_{\theta}(x)$ as $M\rightarrow \infty$. 
To control the variance of the estimate, one can form the running averages $A_n = \tfrac{1}{n}\sum\limits_{i=1}^n\Delta_i$. The results of \cite{joulin} can be used in certain cases to quantify how the variance of the $A_n$ decreases with time.
\section{Examples}\label{sect:example}
\begin{example}\label{nn-example}
We consider a stochastic neural network where at each time only a subset of the edges in the network are activated. 
There are $N$ nodes so that the state space $X$ is $[0,1]^N$. 
The random input is a binary vector in  $\Xi = \{0,1\}^{N\times N}$. 
Let 
$\sigma$
be the sigmoid function
$\sigma(x) = (1+\exp(-x))^{-1}$. 
The function 
$f: X \times \Xi \times \Theta \to X$
is
$$
f_{i}(x,\xi,\theta)
=
\sigma
\left(
u_i(x,\xi,\theta)
\right)
$$
where
$u_i(x,\xi,\theta) = \sum_{k=1}^{n}\xi_{i,k}\theta_{i,k}x_k$. 
The $b_i$ are biases and considered fixed. 
A vector $\xi \in \Xi$ indicates which edges are active at each time step;
The edge $(i,j)$ from $j$ to $i$ is only used if $\xi_{i,j}=1$.
The probability measure on $\Xi$ is defined by 
$\nu(\xi) := \prod\limits_{(i,j)\in E}\rho^{1-\xi_{i,j}}(1-\rho)^{\xi_{i,j}}$.
Under this law, 
in the extreme $\rho=1$ we have $\xi_{i,j} = 0$ for all $i,j$ with probability 1.
The parameter space $\Theta$ is the $N\times N$ matrices $\mathbb{R}^{N\times N}$,
which are the weights $\theta_{i,j}$ between each unit.
We set $A(x) = I$ and $\|\cdot\|_{X} = \|\cdot\|_{\infty}$, hence $d_{A}(x,y) = \|x-y\|_{\infty}$. We set $B(x) = I$. 
We must find conditions so that
  Assumptions \ref{asu:lipVA}, \ref{asu:differentiable}, \ref{asu:lipVB}
  and \ref{asu:f} hold. 
 After setting $\Theta$ to be an arbitrary open ball, the only non-trivial part is the contraction criteria, part \ref{part-ctr-of-l-asu} of Assumption \ref{asu:f}.
Observe that
$
\tfrac{\partial f_i}{\partial x_j}(x,\xi,\theta)
=
\sigma'(u_i(x,\xi,\theta))\xi_{i,j}\theta_{i,j}
.$
With the norm 
$\|\cdot\|_{\infty}$
on $X$ and as $|\sigma'(u)| \leq \tfrac{1}{4}$,
  $$\|
    \tfrac{\partial f}{\partial x}(x,\xi,\theta)
  \|_{\infty} 
  \leq
  \tfrac{1}{4}
  \|\theta\|_{\infty}
    \sup_{i,j}\xi_{i,j}.
  $$
 Note that 
 $
 \left(\int_{\Xi}\left(\sup_{i,j}\xi_{i,j}\right)^2\, \mathrm{d}\nu(\xi)\right)^{1/2} =
 \left(1 - \nu(\xi = 0)\right)^{1/2} =
 \left(1- \rho^{|E|}\right)^{1/2},
 $
 so a sufficient condition for contraction in $d_2$ is
 $\|w\|_{\infty}(1-\rho^{|E|})^{1/2} < 4$.
 The matrix norm induced by
 $\|\cdot\|_{\infty}$
 is the maximum absolute row sum;
 then the condition is that
 the sum of magnitudes of
 incoming weights at each node must 
 be bounded in this way.

The requirements for applying forward sensitivity analysis are met.
For completeness we derive the exact form of the sensitivity system.
The space 
$M$
consists of the linear maps from 
$\mathbb{R}^{N\times N}$
to
$\mathbb{R}^{N}$ and
$
\frac{\partial f_{i}}{\partial \theta_{(j,k)}}(x,\xi,\theta)
=
 \delta_{i,j}\sigma'(u_i(x,\xi,\theta))\xi_{i,k}x_k
 .$
 We use subscripts to denote time, and $v(k)$ means the $k^{th}$ component of vector $v$.
\begin{align*}
  &x_{n+1}(i)
  = \sigma(u(x_n,\xi_{n+1},\theta)(i)) \\
  &m_{n+1}(i,(j,k))\\&\quad\quad=
    \sigma'(u(x_n,\xi_{n+1},\theta)(i))
    \left[
    \delta_{i,j}\xi_{n+1}(i,k)x_{n}(k)
    +
    \sum\limits_{q=1}^{n}\xi_{n+1}(i,q)\theta(i,q)m_n(q,(j,k))
    \right]
\end{align*}
At time $n+1$, node $i$ has to pull from each node $q$ that connects to it the data $m_{n}(q,(j,k))$ and the state variable $x_n(q)$.
\end{example}
\begin{example}
Let 
$\Xi = \mathbb{R}^{2}$
and let
$\nu$ 
be the law of two 
independent random variables $\xi_1,\xi_2$, 
such that $\mathbb{E}[\exp(6|\xi_1|) + |\xi_2|^{2}] < \infty$.
Let $f:\mathbb{R}^{2} \times \Xi\times \Theta \to \mathbb{R}^{2} $ be the function
\begin{equation}\label{ex2-def}
f(x,\xi,\theta)
=
\Big( 
f_{1}(x_1,\xi,\theta)
,
f_{2}(x_1,x_2,\xi,\theta)
\Big)
\end{equation}
where
$
f_{1}(x_1,\xi,\theta) = 
\tfrac{1}{2} x_1 + \theta + \epsilon \xi_1$ and $f_{2}(x_1,x_2,\xi,\theta) = 
\tfrac{1}{2} x_1 x_2 + \epsilon \xi_2
$
Let $g_1,g_2$ be the real valued functions
$g_{1}(x) = \exp(2|x_1|)(1 + |x_2|)$
and
$g_{2}(x) = \exp(2|x_1|)$.
The metric $d_{A}$ will be
defined using the pair $(\|\cdot\|,A)$ 
where 
$\|(u,v)\| = p_1|u| + p_2|v|$ and
$A(x) = g_{1}(x) \oplus g_{2}(x)$,
with
$p_1,p_2$
 determined below.
The parameter $\theta$ is a number and
 $B$ is  
$B(x) = g_1(x)$. 
We seek conditions on $\epsilon$ and $\theta$ that guarantee contraction and the applicability of stochastic forward sensitivity analysis.  We find the following:
\begin{prop}\label{prop:example2-setup}
Let the following hold
\begin{enumerate}
   \renewcommand{\theenumi}{\roman{enumi}}
\item
  The parameter space is
  $\Theta = (-\frac{1}{4}\log 2, \frac{1}{4}\log 2)$,
  \label{ex2-theta-bound}
\item 
  $\epsilon<1$ and
  $
  \left(
    1
    +
    \epsilon \left(\int_{\Xi}|\xi_2|^{2}\, \mathrm{d}\nu(\xi)\right)^{1/2}
  \right)
  \left(
    \int_{\Xi}\exp(2\epsilon|\xi_1|)^{2}\, \mathrm{d}\nu(\xi)
  \right)^{1/2}
  < 2^{1/4}$,
\label{ex2-epsilon-bound}
\item 
  The coefficients
$p_1,p_2$
are any positive numbers such that 
$1+\frac{p_2}{p_1} < 2^{1/4}$.
\label{ex2-pi-bound}
\end{enumerate}
For $\theta\in \Theta$ the stochastic forward sensitivity method is applicable for the system (\ref{ex2-def}).
\end{prop}
\begin{proof}
  See the appendix for a sketch of the calculations involved.
\end{proof}
Based on the definition of $\mathcal{E}^{2}$, the cost functions are those
$e:\mathbb{R}^{2} \to \mathbb{R}$ satisfying
$\sup_{x}|\tfrac{\partial e}{\partial x_{i}}(x)|g_{i}(x)^{-1} <
\infty$ and $ \sup_{x}
|
\tfrac{\partial^{2} e}
{\partial x_{i}\partial x_{j}}(x)
|
g_{i}^{-1}(x)g_{j}^{-1}(x) 
<
 \infty$  for $1 \leq i,j \leq 2$.

Note that since $g_i \geq 1$ the functions in $\mathcal{E}$ include those with 
$\sup_{x}
\|\frac{\partial e}{\partial x}(x)\|
<
\infty$
and
$\sup_{x}
\|
\frac{\partial^{2} e}{\partial x^{2}}(x)
\|
<
\infty$.
The joint process takes the following form. We denote the $k$th component of a vector 
$v$ by $v(k)$,  and use a subscript to denote time.
\begin{align*}
  x_{n+1}(1) 
  &= \tfrac{1}{2} x_n(1) + \theta + \epsilon\xi_{n+1}(1)\\
  x_{n+1}(2)
  &= \tfrac{1}{2} x_n(1)x_n(2) + \epsilon \xi_{n+1}(2) \\
  m_{n+1}(1)
  &= \tfrac{1}{2} m_{n}(1) 
  + 1 \\
  m_{n+1}(2)
  &= \tfrac{1}{2} x_n(2)
  m_{n}(1)
  +
  \tfrac{1}{2} x_n(1)
  m_{n}(2)
\end{align*}
\end{example}
\section{Discussion}\label{sect:discussion}
Our approach to establishing differentiability can be compared with works on \\measure-valued differentiation, such as \cite{heidergott2006measure, heidergott2003taylor}. The ergodicity framework  in those works is based on normed ergodicity \cite{borovkov}, while ours is also based on a norm but involves the derivatives of the cost functions as well. 
The approach to establishing differentiability is based on setting up a certain equation between linear functionals, showing that any solution to that equation must evaluate the stationary derivative, and showing that the equation indeed has a solution. In this sense it is similar to \cite{vazquez1992estimation}, which works with the class of bounded measurable cost functions, and in a different ergodicity framework. The work \cite{pflug-mvd} also used contraction in the Wasserstein distance in an ergodicity framework for stationary gradient estimation.
 
This work was motivated by derivative estimation and optimization in neural networks. The back-propagation procedure is based on \textit{adjoint sensitivity analysis}, as opposed to the forward sensitivity analysis studied here.
 Adjoint sensitivity analysis is often preferred as the auxiliary system in this case evolves in a space which has dimension $n_{X}$ as opposed to $n_{\theta}\times n_{X}$. In \cite{flynn_neco, flynn_mtns} the author analyzed joint gradient estimation/optimization schemes based on adjoint sensitivity analysis. It may be that the methods of this paper can be extended to adjoint sensitivity analysis. A counter example to this possibility would also be very interesting.

Another interesting extension may be to recursively apply the construction to obtain estimators for higher derivatives. Calculating
$
\frac{\partial^{2}}{\partial \theta^{2}}
\mathbb{E}_{\pi_{\theta}}[e(x)]
$
should be equivalent to computing
$
\frac{\partial}{\partial \theta}\mathbb{E}_{\gamma_{\theta}}[g(x)]
$ for the ``cost function''
$
g(x) = \frac{\partial e}{\partial x}(x)m
$. 
\begin{center}{\bf Appendix}\end{center}
\appendix
\subsubsection*{Notations:}
$\Theta$ - space of parameters,
$n_{X}$ - dimensionality of state space for underlying system,
$n_{\Theta}$ - dimensionality of parameter space,
$L(\mathbb{R}^{n},\mathbb{R}^{m})$ - space of linear maps from $\mathbb{R}^{n}$ to $\mathbb{R}^{m}$,
$M$ - the space $L(\mathbb{R}^{n_{\Theta}},\mathbb{R}^{n_{X}})$,
$\|V\|_{L^{p}(\mu)}$ - short hand for 
$\left(\int_{X}\|V(x)\|^{p}\,\mathrm{d}\mu\right)^{1/p}$.
$\mathcal{P}(X)$ - Borel probability measures on $X$,
$\mathcal{P}_{p,V}(X)$ - measures in $\mathcal{P}(X)$ that such that $\|V\|_{L^{p}(\mu)} < \infty$,
$d_{A}$ - metric induced by a Finsler structure,
$\mathcal{P}_{p,A}(X)$ - measures such that $\int_{X}d_{A}(x,x_0)^{p}\,\mathrm{d}\mu(x) < \infty$,
$d_{p,A}$ - Wasserstein distance on the space $\mathcal{P}_{p,A}$,
$\|\cdot\|_{Lip}$ - Lipschitz constant for a function between metric spaces,
$(E \oplus F)(u,v) = (Eu,Fv)$ - direct sum of linear maps; $(E \oplus F)(u,v) = (Eu,Fv)$.
$\|\cdot\|_{A,A}$ - norm for a bilinear map: $\|m\|_{A,A} = \sup_{\|u\|=\|v\|=1}\|A[u,v]\|$, 
$\|\cdot\|_{\mathcal{E}^{2}}$ - the norm 
$\|e\|_{\mathcal{E}^{2}} = 
   \|e\|_{A} + 
   \|\frac{\partial e}{\partial x}\|_{A} + 
   \|\frac{\partial^2 e}{\partial x^{2}}\|_{A,A}$, 
 $I_{n}$ - identity matrix on $\mathbb{R}^{n}$.
\begin{proof}[Proof of Proposition \ref{conditions-for-metric}]
The metric axioms can be shown as in \cite{burago2001course}, Chapter 2.
We show the completeness. 
The condition on $A(x)^{-1}$ means that for some $k$ the inequality
\begin{equation}\label{domination}
\|x - y\| \leq k d_A(x,y)
\end{equation}
holds for all $x,y \in X$.
The continuity of $A$ means that $\|A\|$ is bounded on compact subsets of $X$. Combining this with (\ref{domination}) it follows that $d_A$ and the metric determined on $\|\cdot\|$
are strongly equivalent on compact subsets of $X$.
Using (\ref{domination}) one can show that any $d_A-$Cauchy sequence is contained 
in a compact subset of $X$. By the strong equivalence $d_A$ is complete on this subset.
\end{proof}
\begin{proof}[Proof of Proposition \ref{quadratic-estimate}]
  We will make use of the following inequality:
  Whenever 
  $\gamma : [0,T] \to X$
  is a curve from $x$ to $y$ 
  that is (i) parameterized by arc length,
  and (ii) such that 
  $L(\gamma) \leq d_A(x,y) + \epsilon$, then
  \begin{align}\label{coolineq}
    \int_{0}^{T}d_A(\gamma(t),x)\, \mathrm{d}t \leq \frac{(d_A(x,y) + \epsilon)^{2}}{2}
  \end{align}
  To see this, note that
  for any curve parameterized by arc length,
  $d_A(\gamma(t),x) \leq t$.
  Integrating both sides of this inequality 
  and using the first assumption yields the result.

  We now proceed to the proof of part \ref{quadratic-pointwise}. 
  Let $h:X\to\mathbb{R}^{n}$ 
  be a function satisfying the assumptions of the Proposition.
  Given
  $\epsilon > 0$ ,
  let 
  $\gamma : [0,T] \to X$ be a piecewise $C^{1}$
  curve from $x$ to $y$ with 
  $L(\gamma) \leq d_A(x,y) + \epsilon$.
  Assume that $\gamma$ is parameterized by arc length.
  By the identity $\gamma'(t) = A(\gamma(t))^{-1}A(\gamma(t))\gamma'(t)$, and the assumption on $h$,
  \begin{align*}
    \|h(x) - h(y)\|
    &\leq
      \int_{0}^{T}
      \|
      \tfrac{\partial h}{\partial x}(\gamma(t))\gamma'(t)
      \|\, \mathrm{d}t \\
    &=
      \int_{0}^{T}
      \|
      \tfrac{\partial h}{\partial x}(\gamma(t))
      A(\gamma(t))^{-1}
      A(\gamma(t))\gamma'(t)
      \|\, \mathrm{d}t
    \\&\leq
    \int_{0}^{T}
    B(\gamma(t))
    \, \mathrm{d}t
\end{align*}
Seeing as $B$ is Lipschitz, and invoking inequality (\ref{coolineq}),
\begin{align*}
  &\leq
    \int_{0}^{T}
    \Big(B(x) + \|B\|_{Lip}d_A(\gamma(t),x)\Big)
    \, \mathrm{d}t
    \\&\leq
    B(x)\int_{0}^{T}1\, \mathrm{d}t
    +
    \|B\|_{Lip}\int_{0}^{T}d_A(\gamma(t),x)
    \, \mathrm{d}t 
    \\&\leq
    B(x)[d_A(x,y) + \epsilon]
    +
    \|B\|_{Lip}[ d_A(x,y)^{2}/2 + d_A(x,y)\epsilon + \epsilon^{2}/2]
\end{align*}
Since $\epsilon$ was arbitrary, we have
$\|h(x) - h(y)\|
\leq
B(x)d_A(x,y) + \frac{1}{2}\|B\|_{Lip}d_A(x,y)^{2}$. 

For part \ref{quadratic-measure},
let 
$\gamma$ 
be any coupling of $\mu_1$ with $\mu_2$ such that\\
$
\left(
  \int_{X\times X}d_A(x,y)^{2}\, \mathrm{d}\gamma(x,y)
\right)^{1/2} 
\leq
d_{2,A}(\mu_1,\mu_2) + \epsilon
$.
Then
\begin{align*}
&\left\|\int_{X}h(x)\, \mathrm{d}\mu_1(x) - \int_{X}h(y)\, \mathrm{d}\mu_2(y) \right\| 
\\
&\quad\quad\leq
\int_{X\times X}\|h(x) - h(y)\|\, \mathrm{d}\gamma(x,y) \\
&\quad\quad\leq
\int_{X\times X}
B(x)d_A(x,y)
\, \mathrm{d}\gamma(x,y) 
+
\tfrac{1}{2}\|B\|_{Lip}
\int_{X\times X}
d_A(x,y)^{2}
\, \mathrm{d}\gamma(x,y) \\
&\quad\quad\leq
\|B\|_{L^{2}(\mu_1)}
(d_{2,A}(\mu_1,\mu_2) + \epsilon) + 
\tfrac{1}{2}\|B\|_{Lip}(d_{2,A}(\mu_1,\mu_2) + \epsilon)^{2}
\end{align*}
Since $\epsilon>0$ was arbitrary we are done.

\end{proof}
\begin{proof}[Proof of Proposition \ref{prop:example2-setup}]
We verify
Assumptions 
\ref{asu:lipVA} -
\ref{asu:f}.
For Assumption \ref{asu:lipVA}, the continuity is obvious. As $A$ has a diagonal structure, $\|A(x)^{-1}\| = \max\{ g_{1}(x)^{-1},\,\\g_{2}(x)^{-1}\}$, so it is clear that $\|A(x)^{-1}\| \leq 1$ for all $x$.

For 
Assumption \ref{asu:differentiable}, 
the differentiability is evident. For the integrability, using the basepoint 
$(0,0)$ it suffices that
$\left(
  \int_{\Xi}d_{A}(0,f(x,\xi,\theta))^{2}\, \mathrm{d}\nu(\xi)
\right)^{1/2} < \infty$
for any $(x,\theta) \in X\times\Theta$.
Consider the curve 
$t\mapsto t \, f(x,\xi,\theta)$, for $t \in [0,1]$, from
$0$
to
$f(x,\xi,\theta)$. Then
$
d_{A}(0,f(x,\xi,\theta))
\leq
\int_{0}^{1}\|A(t\,f(x,\xi,\theta))f(x,\xi,\theta)\|\, \mathrm{d}t$.
Next, by definition of $\|\cdot\|$,
\begin{align*}
  &\|A(t\,f(x,\xi,\theta))f(x,\xi,\theta)\|
 = p_1\,|g_1(t\,f(x,\xi,\theta))f_1(x,\xi,\theta)|
    + p_2\,|g_2(t\,f(x,\xi,\theta))f_2(x,\xi,\theta)|
\end{align*}
For the first term on the right hand side of this equation we have
\begin{align*}
&|g_{1}(t\,f(x,\xi,\theta))f_1(x,\xi,\theta)| \\
&\quad\quad=
\exp(2|t\tfrac{1}{2}x_1 + t\theta + t\epsilon\xi_1|)
(1 + |t\tfrac{1}{2}x_1x_2 +t\epsilon\xi_2|)
|\tfrac{1}{2}x_1 + \theta + \epsilon\xi_1| \\
&\quad\quad\leq
\exp(|x_1|)\exp(2|\theta|)\exp(2\epsilon|\xi_1|)
(1+\tfrac{1}{2}|x_1||x_2| + \epsilon|\xi_2|)
  (\tfrac{1}{2}|x_1| + |\theta| + \epsilon|\xi_1|)
\\
&\quad\quad\leq
\exp(2|x_1| + |x_1||x_2|)\exp(2|\theta|)\exp(3\epsilon|\xi_1|)
(1+\epsilon|\xi_2|)
\end{align*}
In the last inequality we used the fact that $\theta < 1/2$.
Likewise, for the second term,
\begin{align*}
|g_2(tf(x,\xi,\theta))f_2(x,\xi,\theta)| &=
\exp(2|t\tfrac{1}{2}x_1 + t\theta + t\epsilon\xi_1|)|\tfrac{1}{2}x_1x_2 + \epsilon\xi_2| \\
&\leq
\exp(|x_1| + |x_1x_2|)\exp(2|\theta|)\exp(2\epsilon|\xi_1|)\epsilon|\xi_2|
\end{align*}
Combining these we obtain a bound for 
$d_{A}(0,f(x,\xi,\theta))$:
\begin{align}
d_{A}(0,f(x,\xi,\theta))
&\leq
p_1\exp(2|x_1| + |x_1x_2|)\exp(2|\theta|)\exp(3\epsilon|\xi_1|)(1+\epsilon|\xi_2|) \nonumber \\
&\quad+
p_2\exp(|x_1| + |x_1x_2|)\exp(2|\theta|)\exp(2\epsilon|\xi_1|)\epsilon|\xi_2|
 \nonumber \\
&\leq (p_1+p_2)\exp(2|x_1| + |x_1x_2|)\exp(2|\theta|)\exp(3\epsilon|\xi_1|)(1+\epsilon|\xi_2|) \label{cool-label}
\end{align}
Let
$Q 
= \left(\int_{\Xi}|\xi_2|^{2}\, \mathrm{d}\nu(\xi)\right)^{1/2}$ 
and set 
$R 
= \left(\int_{\Xi}\exp(2\epsilon|\xi_{1}|)^{2}\,\mathrm{d}\nu(\xi)\right)^{1/2}$.
Squaring and integrating (\ref{cool-label}) yields
\begin{align*}
&\left(\int_{\Xi}d_A(0,f(x,\xi,\theta))^{2}\, \mathrm{d}\nu(\xi)\right)^{1/2} \\
&\quad\leq
(p_1+p_2)\exp(2|x_1|+|x_1x_2|)\exp(2|\theta|)
  \left(\int_{\Xi}\exp(3\epsilon|\xi_1|)^{2}\, \mathrm{d}\nu(\xi)\right)^{1/2} 
  (1 +\epsilon Q)
\end{align*}
which is finite by assumption that $\exp(6|\xi_1|)$ is integrable and that $\epsilon < 1$.

For 
Assumption \ref{asu:lipVB},
the invertibility of
$B(x)$ follows since
$g_{1} > 1$. 
Next, we show
$\|B(x)\|$ 
is Lipschitz for  
$d_{A}$.
Since $\|e\|_{Lip}=\|\tfrac{\partial e}{\partial x}\|_{A}$
when $e$ is differentiable, the Lipschitz continuity of
$g_1$
can be shown as follows.
Let 
$x = (x_1,x_2)$
be a point of differentiability for $(|x_1|,|x_2|)$, and let $p_1|u| + p_2|v| = 1$. Then
\begin{align*}
  |\tfrac{\partial g_1}{\partial x}(x)A(x)^{-1}(u,v)| 
  &= 
    |
    \tfrac{\partial g_1}{\partial x}(x)
    (g_1(x)^{-1}u, g_{2}(x)^{-1}v)
    |
  \\
  &=
    |
    \tfrac{\partial g_1}{\partial x_1}(x)g_1(x)^{-1}u 
    +
    \tfrac{\partial g_1}{\partial x_2}(x)g_2(x)^{-1}v
    | 
\\&\leq
    \max\{
    \tfrac{1}{p_1}
    |\tfrac{\partial g_1}{\partial x_1}(x)g_1(x)^{-1}|,
    \tfrac{1}{p_2}
    |\tfrac{\partial g_1}{\partial x_2}(x)g_2(x)^{-1}|
    \}
\end{align*}
where
$|\tfrac{\partial g_1}{\partial x_1}(x)g_1(x)^{-1}|
\leq
2$
and
$|\tfrac{\partial g_1}{\partial x_2}(x)g_2(x)^{-1}|
\leq
1$.
By an argument using a mollification of $|\cdot|$, this is extended to all points of $X$. Therefore $\|g\|_{Lip} \leq \max\{\tfrac{2}{p_1},\tfrac{1}{p_2}\}$. 
We turn to the functions
$L_{X^{i},\Theta^{j}}$,
starting with $L_{X}$.
Observe the inequalities
\begin{equation}\label{ineq1}
\begin{split}
&g_1(f(x,\xi,\theta))\,
|\tfrac{\partial f_1}{\partial x_1}(x,\xi,\theta)|
\,g_1(x)^{-1} \\
&\quad\leq
\tfrac{1}{2} \exp(2|\theta|)\exp(2\epsilon|\xi_1|)
\exp(|x_1|)
(1+\tfrac{1}{2}|x_1|+\epsilon|\xi_2|)
\exp(-2|x_1|),
\end{split}
\end{equation}
\begin{equation}\label{ineq2}
\begin{split}
&g_{2}(f(x,\xi,\theta))\,
|\tfrac{\partial f_2}{\partial x_1}(x,\xi,\theta)|
\,g_{1}(x)^{-1} 
\\&\quad\leq
\tfrac{1}{2}\exp(2 |\theta|)\exp(2\epsilon|\xi_1|)
\exp(|x_1|)
\exp(-2|x_1|),\quad\quad\quad\quad\quad\quad\quad\quad
\end{split}
\end{equation}
and
\begin{equation}\label{ineq3}
\begin{split}
&g_{2}(f(x,\xi,\theta)) \,
|\tfrac{\partial f_2}{\partial x_2}(x,\xi,\theta)|
\, g_{2}(x)^{-1}
\\&\quad\leq
\tfrac{1}{2}\exp(2 |\theta|)\exp(2 \epsilon|x_1|)
\exp(|x_1|)
|x_1|
\exp(-2|x_1|).\quad\quad\quad\quad\quad\quad
\end{split}
\end{equation}
Next, note that
\begin{align*}
&\|
A(f(x,\xi,\theta))
\tfrac{\partial f}{\partial x}(x,\xi,\theta)
A(x)^{-1}
\| 
\\&\leq 
\max\Big\{
  g_{1}(f(x,\xi,\theta))
  |\tfrac{\partial f_{1}}{\partial x_{1}}(x,\xi,\theta)|
  g_{1}(x)^{-1} 
  +
  \tfrac{p_{2}}{p_1}
  g_{2}(f(x,\xi,\theta))
  |\tfrac{\partial f_{2}}{\partial x_{1}}(x,\xi,\theta)|
  g_{1}(x)^{-1},
  \\&\quad\quad\quad\quad
  g_{2}(f(x,\xi,\theta))
  |\tfrac{\partial f_{2}}{\partial x_{2}}(x,\xi,\theta)|
  g_{2}(x)^{-1}
  \Big\}
\end{align*}
Combining this with the three inequalities (\ref{ineq1}), (\ref{ineq2}), (\ref{ineq3}), we get
\begin{align*}
&\|A(f(x,\xi,\theta))\tfrac{\partial f}{\partial x}(x,\xi,\theta)A(x)^{-1}\|
\\&\leq
\tfrac{1}{2}\exp(2|\theta|)\exp(2\epsilon|\xi_1|)
\exp(|x_1|)
\max\{
  1 + \tfrac{1}{2}|x_1| + \epsilon|\xi_{2}| + \tfrac{p_2}{p_1} 
  ,
  |x_1|
\}
\exp(-2|x_1|) \\
&\leq
\tfrac{1}{2}\exp(2|\theta|)\exp(2\epsilon|\xi_1|)
\exp(|x_1|)
[1 + |x_1| + \epsilon|\xi_{2}| + \tfrac{p_2}{p_1}]
\exp(-2|x_1|)
\end{align*}
Squaring and integrating the right-hand side of the last inequality, and using the independence of the 
$\xi_1$ and $\xi_2$
variables yields
\begin{align*}
  L_{X}(x,\theta)
  &\leq \tfrac{1}{2} 
  \exp(2|\theta|)
  R
  \exp(|x_1|)
  (1 + 
    \epsilon
    Q
     +
    \tfrac{p_2}{p_1} 
    +
    |x_1|)
    \exp(-2|x_1|)
\end{align*}
This is a continuous function of $(x,\theta)$, so the continuity of $L_{X}$ holds. We now show the contraction property.
Using the inequality $a + x \leq a\exp(\frac{x}{a})$ we get
\begin{align*}
&\leq
(1+\epsilon Q + \tfrac{p_2}{p_1})
\tfrac{1}{2}
\exp(2|\theta|)
R 
\exp\left(\left[1  + (1 + \epsilon Q + \tfrac{p_2}{p_1})^{-1}\right]|x_1|\right)
\exp(-2|x_1|)
\end{align*}
Based on this, 
the contraction property holds if
$\epsilon,\theta,p_1,p_2$ are
such that
$
(1+\epsilon Q + \tfrac{p_2}{p_1})
\exp(2|\theta|)
R
<
2$
and one can verify that Assumptions (\ref{ex2-theta-bound}), (\ref{ex2-epsilon-bound}) and (\ref{ex2-pi-bound}) mean that this indeed is the case.
Now consider $L_{\Theta}$. Let
$\|\cdot\|_{\Theta} = |\cdot|$.
Then $
\|
A(f(x,\xi,\theta))
\tfrac{\partial f}{\partial \theta}(x,\xi,\theta)
B(x)^{-1}
\| 
=
g_1(f(x,\xi,\theta))g_1(x)^{-1}.
$
Using a similar analysis as above,
$$
g_{1}(f(x,\xi,\theta))g_{1}(x)^{-1} \leq
\exp(2|\theta|)
\exp(2\epsilon|\xi_1|)
\exp(|x|)(1+\tfrac{1}{2}|x_1| + \epsilon|\xi_2|)\exp(-2|x_1|)
$$
Squaring and integrating the right-hand side of this equation yields
\begin{align*}
L_{\Theta}(x,\theta) &\leq
  \exp(2|\theta|)
 R
  \exp(|x_1|)
  (1+\tfrac{1}{2}|x_1|+\epsilon Q)
  \exp(-2|x_1|) \\&\leq
  (1+\epsilon Q)\exp(2|\theta|)
  R
  \exp(
  [1 + \tfrac{1}{2(1+\epsilon Q)} -2]|x_1|
  )
  \leq
  (1+\epsilon Q)\exp(2|\theta|)
  R
  \end{align*}
From the first inequality we can see that $L_{\Theta}$ is continuous. From the last we can see that $L_{\Theta}$ is bounded on the set $X\times\Theta$.
It remains to verify conditions on the higher derivatives.
The higher derivatives vanish except for 
$\frac{\partial^{2} f}{\partial x^{2}}$.
This is defined as follows
$$
\tfrac{\partial^{2} f_{k}}
     {\partial x_{i}\partial x_{j}}(x,\xi,\theta)
     =
   \begin{cases}
  \tfrac{1}{2} \text{ if } k = 2 \text{ and } i \neq j,\, 0 \text{ otherwise.}
    \end{cases}
$$
For $i=1,2$ we have
$A(x)^{-1}e_i = g_i^{-1}(x)e_i$
and by basic properties of bilinear maps,
\begin{align*}
&A(f(x,\xi,\theta))
\tfrac{\partial^{2}f}{\partial x^{2}}(z)
( A(x)^{-1}e_i , A(x)^{-1}e_j )
= 
A(f(x,\xi,\theta))g_{i}^{-1}(x)g_{j}^{-1}(x)
\tfrac{\partial^{2} f}
     {\partial x_i\partial x_j}(x,\xi,\theta)
\end{align*}
Note that
$\frac{\partial^{2} f}{\partial x_i\partial x_j}(x,\xi,\theta) = 0$
if $i = j$. 
When $i\neq j$  we have
$\frac{\partial^{2} f}
      {\partial x_i\partial x_j}(x,\xi,\theta)
      =
      (0,\tfrac{1}{2})$
and 
$
A(f(x,\xi))
g_{1}^{-1}(x)g_{2}^{-1}(x)
\Big(0,\tfrac{1}{2}\Big) 
= 
\Big(0, g_{2}(f(x,\xi))g_{1}^{-1}g_{2}^{-1}(x)\Big)
.$
Then for any $i,j$, 
$$
\|
A(f(x,\xi,\theta))
\tfrac{\partial^{2} f}{\partial x^{2}}(x,\xi,\theta)
(A(x)^{-1}e_i,
 A(x)^{-1}e_j)
\|
\leq 
p_2g_{2}(f(x,\xi,\theta))g_{1}(x)^{-1}g_{2}(x)^{-1}.$$
Note that
$|g_{1}^{-1}(x)| \leq 1$, and the norms $\|\cdot\|_{1}$ and $\|\cdot\|_{X}$ satisfy
$
\|\cdot\|_{1}
\leq
\max\{\tfrac{1}{p_1},\tfrac{1}{p_2}\}
\|\cdot\|_{X}$.
With this we get
\begin{align*}
  &\|
  A(f(x,\xi,\theta))
  \tfrac{\partial^{2}f}{\partial x^{2}}(x,\xi,\theta)
  ( A(x)^{-1}\oplus A(x)^{-1})
  \|
  \\&\quad\quad \leq
  \left(
    \max\{
      \tfrac{1}{p_1},\tfrac{1}{p_2}
    \}\right)^{2} 
    p_{2}g_{2}(f(x,\xi,\theta))g_{2}(x)^{-1}
    \\&\quad\quad=
    \max\{
       \tfrac{p_2}{p_1^{2}},\tfrac{1}{p_2}
     \}
     g_2(f(x,\xi,\theta))g_{2}(x)^{-1} 
  =\max\{\tfrac{p_2}{p_1^{2}},\tfrac{1}{p_2}\}\exp(2|\theta| + \epsilon|\xi_1|)
\end{align*}
Integrating yields
 $L_{X^{2}}(x,\theta)
  \leq
  \max\{\tfrac{p_2}{p_{1}^{2}},\tfrac{1}{p_2}\}
  \exp(2|\theta|)
  \int_{\Xi}\exp(2\epsilon|\xi_1|)\, \mathrm{d}\nu(\xi)$,\\
which is bounded and continuous on $X\times\Theta$.
\end{proof}
\bibliography{s5}{}
\bibliographystyle{unsrt}
\end{document}